\documentclass[12pt]{article}
\usepackage{amsmath,amsthm,amssymb}
\usepackage{mathrsfs}
\usepackage[titletoc]{appendix}
\usepackage{bbm}
\usepackage{float}
\usepackage{tikz}
\begin{document}
\input amssym.def
\newcommand{\singlespace}{
    \renewcommand{\baselinestretch}{1}
\large\normalsize}
\newcommand{\doublespace}{
   \renewcommand{\baselinestretch}{1.2}
   \large\normalsize}
\renewcommand{\theequation}{\thesection.\arabic{equation}}

\setcounter{equation}{0}
\def \ten#1{_{{}_{\scriptstyle#1}}}
\def \Z{\mathbb{Z}}
\def \C{\mathbb{C}}
\def \R{\mathbb{R}}
\def \Q{\mathbb{Q}}
\def \N{\mathbb{N}}
\def \F{\mathbb F} 
\def \L{\Bbb L}
\def \D{\mathbb{D}}
\def \l{\lambda}
\def \V{V^{\natural}}
\def \wt{{\rm wt}}
\def \tr{{\rm tr}}
\def \Res{{\rm Res}}
\def \End{{\rm End}}
\def \Aut{{\rm Aut}}
\def \mod{{\rm mod}}
\def \Hom{{\rm Hom}}
\def \im{{\rm im}}
\def \<{\langle}
\def \>{\rangle}
\def \w{\omega}
\def \c{{\tilde{c}}}
\def \o{\omega}
\def \t{\tau }
\def \ch{{\rm ch}}
\def \a{\alpha }
\def \b{\beta}
\def \e{\epsilon }
\def \la{\lambda }
\def \om{\omega }
\def \O{\Omega}
\def \qed{\mbox{ $\square$}}
\def \pf{\noindent {\bf Proof: \,}}
\def \voa{vertex operator algebra\ }
\def \voas{vertex operator algebras\ }
\def \p{\partial}
\def \1{{\bf 1}}
\def \ll{{\tilde{\lambda}}}
\def \H{{\bf H}}
\def \h{{\frak h}}
\def \g{{\frak g}}
\def \rank{{\rm rank}}
\def \({{\rm (}}
\def \){{\rm )}}
\def \Y {\mathcal{Y}}
\def \I {\mathcal{I}}
\def \A {\mathcal{A}}
\def \B {\mathcal {B}}
\def \Cc {\mathcal {C}}
\def \H {\mathcal{H}}
\def \M {\mathcal{M}}
\def \V {\mathcal{V}}
\def \O{{\bf O}}
\def \AA{{\bf A}}
\def \1{{\bf 1}}
\def\Ve{V^{0}}
\def\ha{\frac{1}{2}}
\def\se{\frac{1}{16}}
\def\Irr{\rm Irr}
\def\Ind{{\rm Ind}}
\def\o{\omega}
\newcommand{\BF}{\mathbb{F}}
\newcommand{\BQ}{\mathbb{Q}}
\newcommand{\BC}{\mathbb{C}}
\newcommand{\BL}{\mathbb{L}}
\newcommand{\BR}{\mathbb{R}}
\newcommand{\BZ}{\mathbb{Z}}

\newcommand{\BN}{\mathbb{N}}
\def \1{{\bf 1}}
\singlespace
\newtheorem{thm}{Theorem}[section]
\newtheorem{prop}[thm]{Proposition}
\newtheorem{lem}[thm]{Lemma}
\newtheorem{cor}[thm]{Corollary}
\newtheorem{rem}[thm]{Remark}
\newtheorem*{CPM}{Theorem}
\newtheorem{definition}[thm]{Definition}
\newtheorem{de}[thm]{Definition}

\begin{center}
{\Large {\bf Cliffold algebras,  modular Virasoro vertex operator algebras and $\Z[\ha]$-forms}} \\

\vspace{0.5cm} Chongying Dong\footnote{Supported by the Simons Foundation 634104 and NSFC grant 1187135}\\
Department of Mathematics, University of
California, Santa Cruz, CA 95064 
\vspace{0.5cm}

Ching Hung Lam\footnote{Partially supported by MoST grant 104-2115-M-001-004-MY3 of Taiwan}\\
Institute of Mathematics, Academia Sinica, Taipei 10617, Taiwan\\

\vspace{0.5cm}
Li Ren\footnote{Supported by NSFC grant 12071314}\\
 School of Mathematics, Sichuan University, 
Chengdu 610064, China\\
\end{center}
\hspace{1.5 cm}

\begin{abstract}
This paper consists of two parts: (1)  Using a $\Z[\ha]$-form of Virasoro vertex operator algebra $L(\ha,0)$ with central charge $\ha$,  we obtain a modular vertex operator algebra over any field $\F$ of finite characteristic different from 2.
We determine the generators and classify the irreducible modules for this vertex operator algebra. (2) We investigate 
modular framed vertex operator algebras. In particular, the rationality of  modular framed vertex operator algebras is established. For a modular code vertex operator algebra,  the  irreducible modules are constructed and classified.  Moreover, a $\Z[\frac{1}{2}]$-form for any framed vertex operator algebra
over $\C$ is constructed. As a result, one can obtain a modular  framed vertex operator algebra from any  framed vertex operator algebra 
over $\C.$ 
\end{abstract}

\section{Introduction}

This paper is a continuation of \cite{DR2}.  We  study modular Virasoro vertex operator algebra $L(\ha,0)_\F$
 over any field $\F$ with $\ch\F\ne 2$ based on a $\Z[\ha]$-form of Virasoro vertex operator algebra $L(\ha,0)$ over $\C,$ and develop a general theory of modular framed vertex operator algebra.

The Virasoro vertex operator algebra  $L(\frac{1}{2},0)$ over $\C$ and its representation theory including the fusion rules \cite{DMZ, W1} are the foundation of the framed vertex operator algebras \cite{M1, M2, DGH}. It is well known that  $L(\frac{1}{2},0)$ is rational and has exactly three irreducible modules  $L(\frac{1}{2},h)$ for $h=0, \frac{1}{2}, \frac{1}{16}.$ The fusion products has the following simple form
$$ L(\frac{1}{2},0)\times L(\frac{1}{2},h)=L(\frac{1}{2},h),\quad  L(\frac{1}{2}, \frac{1}{2}) \times L(\frac{1}{2}, \frac{1}{2})= L(\frac{1}{2}, 0 ), $$
$$ L(\frac{1}{2}, \frac{1}{2}) \times  L(\frac{1}{2}, \frac{1}{16})= L(\frac{1}{2}, \frac{1}{16}),
\quad   L(\frac{1}{2}, \frac{1}{16})
 \times  L(\frac{1}{2}, \frac{1}{16})=L(\frac{1}{2},0)+ L(\frac{1}{2}, \frac{1}{2}).$$

Modular vertex operator algebra $L(\ha,0)_\F$ associated to the irreducible highest weight module for the Virasoro algebra over any field of $\ch \F\ne2,7$ has been investigated in \cite{DR1,DR2}.  It was proved that   $L(\ha,0)_\F$ is a rational vertex operator algebra with three inequivalent irreducible modules $L(\frac{1}{2},h)_\F$ for $h=0, \frac{1}{2}, \frac{1}{16},$ and the character of $L(\ha,h)_\F$ is also the same as in the complex field case. Moreover $L(\ha,0)_\F$   has  the same fusion rules as  before. These results lead to a theory of modular framed vertex operator algebra as presented in the second part of this paper.  While the assumption $\ch \F\ne 2$ assures that the Virasoro algebra relation still makes sense, the assumption $\ch \F \ne 7$ avoids $\ha=\se.$ In fact, if $\ch \F=7,$ $L(\frac{1}{2},h)_\F$ has two inequivalent irreducible modules $L(\frac{1}{2},h)_\F$ with  $h=0, \frac{1}{2}$ \cite{DR2}. But the structure and representations of $L(\frac{1}{2},h)_\F$ are not well understood in this case. 
 
The reason for us to choose $L(\ha,0)$  only is  that $L(\ha,0)$ has a fermionic realization which gives a $\D$-form where $\D=\Z[\ha].$  Let  $H_\D=\D a$ be a rank one free $\D$ module with a non-degenerate bilinear form such that $(a,a)=1.$ There is a canonical vertex operator superalgebra $V(H_\D,\Z+\ha)$  over $\D$ and  its twisted module $V(H_\D,\Z).$ For
  $Z=\Z+\ha, \Z,$ $V(H_\D,Z)=V(H_\D,Z)_{\bar 0}\oplus V(H_\D,Z)_{\bar 1}$ where $V(H_R,Z)_{\bar 0}$ is the even part, $V(H_\D,Z)_{\bar 1}$ is the odd part of $V(H_\D,Z).$  It is well known that $V(H_\C,\Z+\ha)_{\bar 0}$ is isomorphic to $L(\ha,0),$ $V(H_\C,\Z+\ha)_{\bar 1}$ is isomorphic to $L(\ha,\ha),$ and both  $V(H_\C,\Z)_{\bar i}$ is isomorphic to
$L(\ha,\se)$ for $i=0,1$ \cite{KR}. Then $V(H_\D,Z)_{\bar i}$ is a $\D$-form of $V(H_\C,Z)_{\bar i}$ for $Z=\Z+\ha, \Z$ and $i=0,1.$ It is easy to show that $V(H_\F,Z)_{\bar i}=\F\otimes V(H_\D,Z)_{\bar i}.$ If $\ch\F\ne 7,$ $L(\ha,0)_\F=V(H_\F,Z+\ha)_{\bar 0},$ and $L(\ha,\ha)_\F=V(H_\F,\Z+\ha)_{\bar 1},$ $L(\ha,\se)_\F=V(H_\F,\Z)_{\bar i}$ for $i=0,1$ \cite{DR2}.
But if $\ch\F=7,$ $V(H_\F,\Z+\ha)_{\bar i}$ is a direct sum of infinitely many irreducible modules for the Virasoro algebra \cite{DR2}.  

The first part of  this paper is devoted to the study of  vertex operator algebra $V(H_\F,\Z+\ha)_{\bar 0}$ instead of $L(\ha,0)_\F$ as  $V(H_\F,\Z+\ha)_{\bar 0}$ has a better representation theory. We prove that  $V(H_\F,\Z+\ha)_{\bar 0}$ is a simple vertex operator algebra  generated by  two  vectors $\omega, J$ and classify the irreducible $V(H_\F,\Z+\ha)_{\bar 0}$-modules  which are isomorphic to $V(H_\F,Z)_{\bar i}$ for $Z=\Z+\ha, \Z,$ $i=0,1$ and $V(H_\F,\Z)_{\bar 0}\cong V(H_\F,\Z)_{\bar 1}$ for all $\F.$  If $\F\ne 7,$ $J$ can be generated from $\omega.$ The classification of irreducible modules uses the associative algebra $A(V(H_\F,\Z+\ha)_{\bar 0})$ as studied in \cite{Z, DR1}. We show that if $\ch\F=7$ then $A(V(H_\F,\Z+\ha)_{\bar 0})$ is isomorphic to $\F[x]/I$ where $I$ is an idea generated by
$x(x-1)(x-3)$ and thus conclude that $V(H_\F,\Z+\ha)_{\bar 0})$ has exactly three inequivalent irreducible modules  
(cf. Theorem \ref{t5.2}). 
We believe that $V(H_\F,\Z+\ha)_{\bar 0}$ is rational, but we could not establish this claim in the paper. 

The second part  of the paper  deals with modular framed vertex operator algebras. Based on the representation theory
of $L(\ha,0),$ the framed vertex operator algebras over $\C$ have been studied extensively in \cite{DMZ, M1, M2, DGH, LY2}. In this paper, we investigate 
framed vertex operator algebras  over any algebraically closed field $\F$ whose characteristic is different from $2$ and $7.$  Specifically, we  determine the structure of   a modular  framed vertex operator algebra in terms of binary codes and establish the rationality.  We also classify the irreducible modules for a modular code vertex operator algebra.  In addition, we obtain a modular framed vertex operator algebra  from any framed vertex operator algebra over $\C$ by constructing a $\Z[\frac{1}{2}]$-form. 

The study of framed vertex operator algebras over $\C$  was motivated by discovering that the moonshine vertex operator
algebra $V^{\natural}$ \cite{FLM} contains a vertex operator subalgebra $L(\frac{1}{2},0)^{\otimes 48}$ with the same central charge in \cite{DMZ}. A systematic investigation of framed vertex operator algebras  was given in \cite{M2, DGH}.  
A framed vertex operator algebra $V$ over $\C$ contains  a rational vertex operator subalgebra $L(\frac{1}{2},0)^{\otimes r}$ where $r/2$ is the central charge of $V.$ The main idea in the theory of framed vertex operator algebra is to decompose the vertex operator algebra $V$ into
a direct sum of finitely many irreducible  $L(\frac{1}{2},0)^{\otimes r}$-modules. One can then use the fusion rules for   $L(\frac{1}{2},0)$ to study this decomposition. It turns out that this decomposition is very powerful in understanding both structure and representation theory of $V.$ This leads to a better understanding of $V^{\natural}.$ In particular, $V^{\natural}$ is holomorphic \cite{D}, and  two weak versions of the Frenkel-Lepowsky-Meurman's conjecture on  uniqueness of the $V^{\natural}$ have been given in \cite{DGL, LY1}.  A new construction of $V^{\natural}$ has been obtained in \cite{M3}.   The theory of framed vertex operator algebra over $\C$ also plays important rules in the study of  holomorphic vertex operator algebras with $c=24.$ Many holomorphic vertex operator algebras with $c=24$ are framed vertex operator algebras \cite{LS}.

The structure of  a framed vertex operator algebra $V$ over $\F$ is similar to that over $\C.$ First, for every linear binary even code $C\subset\Z_2^r$, 
there is a code vertex operator algebra 
$$M_C=\bigoplus_{x=(x_1,..,x_r)\in C}L(\frac{1}{2},\frac{x_1}{2})_\F\otimes \cdots \otimes L(\frac{1}{2},\frac{x_r}{2})_\F$$
associated to $C.$ Code vertex operator algebra $M_C$ is a special class of framed vertex operator algebras which has no $ L(\frac{1}{2},\frac{1}{16})_\F$ involved.  The representation theory of $M_C$ over $\C$ is well understood  due to the work in \cite{M1, M2, LY2}.  We can associate two binary even codes $C$ and $D$ to $V$ such that the code vertex operator algebra $M_C$ is a subalgebra  of  $V$ and 
$$V=\oplus_{d\in D}V^d$$ where $V^0=M_C$ and each $V^d$ is a simple current as $M_C$-module. Using this decomposition, we can show that $V$ is rational.  It is worthy to mention that the decomposition of $V$ into 
irreducible $M_C$-modules is relatively easy over $\C$ as the minimal weights of $L(\frac{1}{2},h)_\C$ are obvious. But one needs extra effort to understand how to put $L(\frac{1}{2},\frac{x_1}{2})_\F\otimes \cdots \otimes L(\frac{1}{2},\frac{x_r}{2})_\F$ in $V$ whose $\Z$-gradation is not given by the weights anymore. 

A classification of irreducible modules for an arbitrary framed vertex operator algebra seems difficult at this stage. This has not been carried out completely over $\C.$ As in \cite{M1, M2},  we can classify irreducible modules for any code vertex operator algebra $M_C.$ Although the main idea is similar to that given in \cite{M1, M2}, we adopt a different approach. The main tool we use in this paper is the vertex operator superalgebra $V(H_\F)$ associated to the infinite dimensional Clifford algebra and its
twisted modules $V(H_\F,d)$ for any codeword $d\in \Z_2^r$ where $H_\F$ is $r$-dimensional vector space with a nondegenerate bilinear form.  Since $M_C$ is a vertex operator subalgebra of  $V(H_\F),$
we show that $V(H_\F,d)$ is a completely reducible $M_C$-module
if $d\in C^{\perp}$ and any irreducible $M_C$-module is obtained in this way.

Constructing a $\Z[\frac{1}{2}]$-form for a framed vertex operator algebra over $\C$ is the key to produce a modular framed vertex  operator algebra from a framed vertex operator algebra over $\C.$
It is easy to see that  for any integral domain $\D$ and  a free $\D$-module $H_\D$ of rank $r,$ $V(H_\D)$ is a vertex operator 
superalgebra and  $V(H_\D,d)$ is its twisted module.
If $\D=\Z[\frac{1}{2}]$,  we have natural $\Z[\frac{1}{2}]$-forms  $V(H_{\Z[\frac{1}{2}]})$ of $V(H_\C)$, $V(H_{\Z[\frac{1}{2}]},d)$ of $V(H_\C,d).$ 
Consequently we obtain a $\Z[\frac{1}{2}]$-form $(M_C)_{\Z[\frac{1}{2}]}$ of framed vertex operator algebra $M_C$ over $\C$ and a $\Z[\frac{1}{2}]$-form
for any irreducible $M_C$-module which is a simple current.  Although  we could not prove in this paper that any irreducible
$M_C$-module has a $\Z[\frac{1}{2}]$-form, the explicit construction of $\Z[\frac{1}{2}]$-form of any simple current for $M_C$ over $\C$ 
is good enough for us to obtain a $\Z[\frac{1}{2}]$-form for any framed vertex operator algebra over $\C.$ 

We should mention that we use a lot of ideas and techniques developed in \cite{M1, M2, DGH, LY2} for dealing framed vertex operator algebras over $\C$ in this paper although the treatments are more complicated.

There are other important work on modular vertex operator algebras. We refer the readers to \cite{B2, BR} for the modular moonshine, \cite{DR1, R} for the representation theory 
of modular vertex operator algebras, \cite{DG1, DG2, DG3, C, Mc} for integral forms of vertex operator algebras,
\cite{AW, JLM,LM1,LM2,LM3,Mu,W2} for modular affine vertex operator algebras, modular Virasoro vertex operator algebras, modular Heisenberg vertex operator algebras and modular lattice vertex operator algebras.

The paper is organized as follows. Section 2 is a review of  basics of  vertex operator superalgebras and their  twisted modules over  an integral domain. Section 3 gives a construction  of $L(c,h)_\F$ and related results.
In Section 4, we study modular vertex operator algebra $V(H_\F,\Z+\ha)_{\bar 0}$ and its modules, and prove that both $V(H_\F,\Z+\ha)$ and $V(H_\F,\Z)$ are modules for the restricted Virasoro algebra. 
 Section 5 is devoted to the study of generators of $V(H_\F,\Z+\ha)_{\bar 0}.$ In particular, we show that  $V(H_\F,\Z+\ha)_{\bar 0}$ is generated by two vectors. In Section 6, we give a spanning set of  $V(H_\F,\Z+\ha)_{\bar 0}.$  We determine the associative algebra $A( V(H_\F,\Z+\ha)_{\bar 0})$ and classify the irreducible  $V(H_\F,\Z+\ha)_{\bar 0}$-modules in Section 7. We investigate the structure of framed vertex operator algebras over any algebraically closed field $\F$ with $\ch\,\F\neq 2, 7$ in Section 8. As in the case $\F=\C$, we can associate two even binary codes $C$ and $D$ to a framed vertex operator algebra $V$ over $\F.$ These two codes play crucial roles in studying 
the structure and representation theory.  We discuss the vertex operator superalgebra $V(H_\D)$ and its twisted modules
$V(H_\D,d)$ over any integral domain $\D$ in Section 9. In the case $\D=\F$, we write down explicit decompositions of
$V(H_\F,d)$ into a direct sum of irreducible $L(\frac{1}{2},0)_\F\otimes \cdots\otimes L(\frac{1}{2},0)_\F$-modules
$L(\frac{1}{2},h_1)_\F\otimes\cdots\otimes  L(\frac{1}{2},h_r)_\F.$ The code vertex operator algebra $M_C$ is studied in Section 10. We decompose each $V(H_\F,d)$ into a direct sum of irreducible $M_C$-modules.  For $\D=\Z[\frac{1}{2}]$, we construct a $\D$-form for the code vertex operator algebra $M_C$ over $\C$ and a $\D$-form for any irreducible $M_C$-module which is a simple current  in Section 11. The main idea is to use the standard $\D$-form $V(H_\D,d)$ of $V(H_\C,d).$ In Section 12, we construct intertwining operators
among  $L(\frac{1}{2}, 0)_\D$-modules $L(\ha,h)_\D$ for $\D=\Z[\frac{1}{2}]$.  These results are used in Section 13 to construct a $\D$-form for any framed vertex operator algebra
over $\C.$ 

In this paper, we always assume that $\F$ is an algebraically closed field although some results still hold without this assumption.

\section{Basics}
\setcounter{equation}{0}

We review the basics of  vertex operator superalgebras and their twisted modules  over an integral domain $\D$ with $\ch \D\ne 2$ and $\frac{1}{2}\in \D$ from  \cite{B1, DR1, DR2} (also see  \cite{DL, DLM3, FFR, LL, X, DG1, Mc}).

A super $\D$-module is a $\Bbb Z_{2}$-graded free $\D$-module
 $V=V_{\bar{0}}\oplus V_{\bar{1}}$ such that both $V_{\bar{0}}$ and $V_{\bar{1}}$ are free $\D$-submodules.
 As usual  we let $\tilde{v}$ be $0$ if $v\in V_{\bar{0}}$, and $1$ if  $v\in V_{\bar{1}}$.

A  vertex operator superalgebra $V=(V,Y,1,\omega)$ over $\D$ is a
$\frac{1}{2}\Bbb Z$-graded super $\D$-module
$$V=\bigoplus_{n\in{ \frac{1}{2}\Bbb Z}}V_n= V_{\bar{0}}\oplus V_{\bar{1}}$$
with  $V_{\bar{0}}=\sum_{n\in\Z}V_n$ and
$V_{\bar{1}}=\sum_{n\in\frac{1}{2}+\Z}V_n$
such that the axioms of vertex operator superalgebra over $\C$ hold  and  for $n\in\Z,$ $s,t\in \ha\Z,$ $u\in V_s, v \in V_t,$
$u_nv\in V_{s+t-n-1}$, 
where $u_n\in \End V$ is given by  $Y(u,z)=\sum_{n\in \Z} u_n z^{-n-1}$.  
Here we assume that the central charge $c$ of the Virasoro algebra lies in $\D.$
If $v\in V_s$, we will call $s$ the degree of $v.$ We also have the notion of
vertex operator algebra $V$ over $\D$ if $V_{\bar 1}=0.$

An automorphism $g$ of a vertex operator superalgebra $V$ is a $\D$-module automorphism of $V$ such that $gY(u,z)v=Y(gu,z)gv$ for all $u,v\in V,$ $g\1=\1,$ $g\omega=\omega$ and $gV_n=V_n$ for all $n\in \ha \Z.$ It is clear that any automorphism preserves
$V_{\bar 0}$ and $V_{\bar 1}.$ There is a special automorphism  $ \sigma$  such that $ \sigma|V_{\bar 0}=1$
and $\sigma|V_{\bar 1}=-1.$ We see that $\sigma$ commutes with any automorphism.

Fix $g\in \Aut(V)$ of order $T<\infty.$  We assume that $\frac{1}{T}\in \D$ and $\D$ contains a primitive $T$-th root of unity $\eta.$ Then $V$ decomposes into eigenspaces of $g$:
$$V=\oplus_{r\in \Z/T\Z}V^{r}$$
where $V^r=\{v\in V|gv=\eta^rv\}.$
Then we have the notion of weak, admissible  $g$-twisted $V$-module $M=(M,Y_M)$ \cite{DZ, DR1} where $M=M_{\bar 0}\oplus M_{\bar 1}$ is $\Z_2$-graded.

If $\D$ is a field, a vertex operator algebra $V$ is called $g$-rational if any admissible $g$-twisted $V$-module is completely reducible. $V$ is rational if $V$ is $1$-rational.

\begin{thm}\label{t2.1} If $V$ is a rational vertex operator algebra over an algebraically closed field $\F$ with $\ch\F\ne 2$  then there are only finitely many inequivalent irreducible admissible modules and the homogeneous subspaces of any irreducible admissible module are finite dimensional.
\end{thm}

Theorem \ref{t2.1} in the case $\F=\C$ was obtained in \cite{DZ}, and in the case  that $\ch\F$ is positive was given in \cite{DR1, R}.  The proof of Theorem \ref{t2.1} is similar to those given in \cite{DZ, DR1,R}. 

 The following result was obtained in \cite{FHL} in the case $\F=\C.$ The same proof works here.
 \begin{lem}\label{tensorrational}  Let $V^i=(V^i, Y^i,\1_i,\omega_i)$ be vertex operator superalgebras over $\F$, $g_i$ be an automorphism of $V^i$ of finite order, and $M^i=(M^i,Y_i)$ be an admissible $g_i$-twisted $V^i$-module for $i=1,...,n.$ 
 
 (1)  $V=V^1\otimes \cdots \otimes V^n$ is a vertex operator superalgebra, $g=g_1\otimes \cdots \otimes g_n$ is an automorphism of $V$ of finite order, and 
 $M=M^1\otimes \cdots \otimes M^n$ is an admissible $g$-twisted $V$-module in an obvious way. 
 
 (2) $M$ is irreducible if and only if each $M^i$ is irreducible. Moreover, every irreducible $V$-module is obtained in this way.
 
 (3) If $V^i$ is $g_i$-rational for all $i$ then $V$ is $g$-rational.
 \end{lem}

We need  the notion of contragredient modules from \cite{FHL, X, Y, DR1}. 
Let $V$ be a vertex operator algebra over a field $\F$. 
Let $g$ be an automorphism of $V$ of order $T<\infty$ and let $M=\oplus_{n\geq 0}M(n)$ be an admissible $g$-twisted $V$-module.
We define the graded dual $M'$ of $M$ as
$$M'=\oplus_{n\geq 0}M(n)^*$$
where $M(n)^*=\Hom_\F(M(n),\F).$
We denote the natural pair from $M'\times M\to \F$ by $(,).$
Also assume that the operators $\frac{L(1)^n}{n!}$ make sense on $V$ for $n\geq 0.$
Then $M'=(M',Y)$ is an admissible $g^{-1}$-twisted $V$-module such that
$$(Y(v,z)w',w)=(w',Y(e^{zL(1)}(-1)^{\deg v+2(\deg v)^2+r}z^{-2\deg v}v,z^{-1})w)$$
for any $v\in V_{\bar r}$, $w'\in M'$ and $ w\in M.$  The reason for us to use $(-1)^{\deg v+2(\deg v)^2+r}$ instead of 
$(e^{\pi i})^{\deg v}$ is that we do not assume the square root of $-1$ is contained in $\F.$ 
Moreover, if each homogeneous subspace of $M(n)$ is finite dimensional, then $M'$ is irreducible if and only if $M$ is irreducible.

Next we define intertwining operators and fusion rules among admissible $g_k$-twisted
modules $(M_k,Y_k)$ for $k=1,2,3$ where $g_k$ are commuting
automorphisms of order $T_k.$   In the case that $\F=\C,$ the definitions are the same as in \cite{X}. We now assume that  $\ch \F=p$ is a prime.
Let $T$ be the least common multiple of $T_1,T_2,T_3$ and assume $T\ne 0$ in $ \F$ and $\F$ contains a $T$-th  primitive root of the unity $\eta.$
In this case
$V$ decomposes into the direct sum of common
eigenspaces
$$V=\oplus_{j_1,j_2=0}^{T-1}V^{(j_1,j_2)}$$
where
$$V^{(j_1,j_2)}=\{v\in V\mid g_kv=\eta^{j_k}v, k=1,2\}.$$
 As usual, we denote the localization of $\Z$ at  the prime ideal $(p)=p\Z$ by $\Z_{(p)}.$ That is, $\Z_{(p)} =\{\frac{m}{n}\in \BQ\mid p\nmid n\}.$
 Assume that
$L(0)|_{M_i(n)}=\lambda_i+n$ for all $n$ where $\lambda_i\in \Z_{(p)}$ which is understood to be a number in $\BF.$

An \emph{intertwining operator} of type $\left(\begin{array}{c}
M_3\\
M_{1\ }M_{2}
\end{array}\right)$ is a linear map
\[
I(\cdot,\ z):\ M_{1}\to\text{\ensuremath{\mbox{Hom}(M_{2},\ M_{3})\{z\}}}
\]

\[
\ \ \ \ \ \   u\mapsto I(u,\ z)=\sum_{n\in\frac{1}{T}\mathbb{Z}}u(n)z^{-n-1-\lambda_1-\lambda_2+\lambda_3}
\]
 satisfying:

(1) for any $u\in M_{1}$ and $v\in M_{2}$, $u(n)v=0$ for $n$
sufficiently large;

(2) $I(L(-1)v,\ z)=(\frac{d}{dz})I(v,\ z)$;

(3) $u(m)M_2(n)\subset M_3(n-m-1+\deg u)$ for $m,n\in \frac{1}{T}\BZ;$

(4) for any $u\in V_{\bar s}^{(j_1,j_2)},\ v\in ( M_1)_{\bar t},$
\[
z_{0}^{-1}\left(\frac{z_1-z_2}{z_0}\right)^{j_1/T}\delta\left(\frac{z_{1}-z_{2}}{z_{0}}\right)Y_{3}(u,\ z_{1})I(v,\ z_{2})
\]
\[-(-1)^{st}z_{0}^{-1}\left(\!\frac{-z_2\!+\!z_1}{z_0}\!\right)^{j_1/T}\delta\left(\frac{-z_{2}+z_{1}}{z_{0}}\right)I(v,\ z_{2})Y_{2}(u,\ z_{1})
\]
\[
=z_{2}^{-1}\left(\frac{z_1-z_0}{z_2}\right)^{-j_2/T}\delta\left(\frac{z_{1}-z_{0}}{z_{2}}\right)I(Y_{1}(u,\ z_{0})v,\ z_{2}).
\]
We denote the space of intertwining operators of type $\left(\begin{array}{c}
M_{3}\\
M_{1\ }M_{2}
\end{array}\right)$
by $I\left(\begin{array}{c}
M_{3}\\
M_{1\ }M_{2}
\end{array}\right).$ We call  $N_{M_1\,M_2}^{M_3}=\dim I\left(\begin{array}{c}
M_{3}\\
M_{1\ }M_{2}
\end{array}\right)$ the fusion rules.
It is proved in \cite{DR1} that the fusion rules $N_{M_1\,M_2}^{M_3}$ are independent of the choices of $\lambda_i.$
Note that  if $N^{M_3}_{M_1M_2}>0$ then $g_3=g_1g_2$ (see [X]). So we now assume that $g_3=g_1g_2.$
This definition of intertwining operator here is the same as that given in \cite{DR2} when $g_i=1$ for all $i.$

We also remark that if $\F=\C$ we do not need $\Z_{(p)}$ in the definition, see \cite{X, DLM0}. 

The following result is a generalization of Proposition 11.19 of \cite{DL} with the same proof by noting that the commutativity and associativity still hold in the current situation.
\begin{lem}\label{dl} 
Let $M_i$ and $I$ be as before. We also assume that $M_1, M_2$ are irreducible and $I(u,z)v=0$ for some nonzero $u\in M_1, v\in M_2.$ Then
$I=0.$
\end{lem}

The fusion rules have certain symmetry properties.
\begin{lem} Let $M_k$ be as before. 

(1) We have $N^{M_3}_{M_2M_1}=N^{M_3}_{M_1M_2}.$

(2) If each homogeneous subspaces of $M_2$ and $M_3$ are finite dimensional and $\frac{L(1)^n}{n!}$ are well defined on $V$ for $n\geq 0,$
then $N^{M_3}_{M_2M_1}=N^{M'_2}_{M_1M'_3}.$
\end{lem}

Let $V^i$  be rational vertex operator algebras over $\F$ and let $M^i_{j}$ be irreducible admissible  $V^i$-modules for $i=1,...,n$ and $j=1,2,3.$  We also assume that $N^{M^i_3}_{M^i_1 M^i_2}$ is finite for all $i.$  Using the exact proof given in \cite{DMZ} in the case $\F=\C$,  we obtain the following result.
\begin{lem}\label{tensorfusion}  The fusion rule  $N^{M^1_3\otimes \cdots \otimes M^n_3}_{M^1_1\otimes \cdots \otimes M^n_1\  M^1_2\otimes \cdots \otimes M^n_2}$ is equal
to $\prod_{i=1}^nN^{M^i_3}_{M^i_1 M^i_2}.$ 
 \end{lem}

We now formulate the notion of tensor product $M_1 \boxtimes M_2$ (which is also called the fusion product and denoted by
$M_1\times M_2$)
of $M_1$ and $M_2$: it is an admissible $g_3$-twisted $V$-module
defined by a universal mapping property (\cite{L3, HL1, HL2}):
 {\it A tensor product} for the ordered pair $(M_1,M_{2})$ is a pair
$(M,F)$ consisting of a weak $g_3$-twisted
$V$-module $M$ and an intertwining
operator $F$ of type
$\left(\begin{array}{c}M\\M_{1}\,M_{2}\end{array}\right)$ such that
the following universal
property holds: for any weak $g_3$-twisted $V$-module $W$ and any intertwining
operator $I(\cdot,z)$ of type
$\left(\begin{array}{c}W\\M_{1}\,M_2\end{array}\right)$, there
exists a unique $V$-homomorphism $\psi$ from $M$ to $W$ such that
$I(\cdot,z)=\psi\circ F(\cdot,z)$. (Here $\psi$ extends canonically to a
linear map from $M\{z\}$ to $W\{z\}$.)

It is easy to show that if $V$ is $g_k$-rational for $k=1,2,3$, then  $M_1 \boxtimes M_2$ exists and is equal to
$\oplus_{M_3\in {\cal M}(g_3)}N_{M_1,M_2}^{M_3}M_3$ where ${\cal M}(g_3)$ is the set of inequivalent irreducible $g_3$-twisted $V$-modules.

For studying the modular framed vertex operator later, we need the following definition:
\begin{definition} A $\D$-form of  a vertex superalgebra $V$ over $\C$  is a $\D$-submodule $I$ of $V$  such that $(I, Y, \1, \omega)$ is a vertex operator superalgebra over $\D$ and for each $n$, $I_n=I\cap V_n$ is a $\D$ form of $V_n$ for all $n$
in the sense that $I_n$ is a free $\D$-module whose rank is equal to the dimension of $V_n.$
\end{definition}

Unlike \cite{DG1}, we do not assume that there is a nondegenerate symmetric bilinear form on $V.$ Also note that $I\cap V_{\bar{0}}$ is a $\D$ vertex operator algebra form for $V_{\bar{0}}.$

\begin{definition} Let $V$ be a vertex operator superalgebra and $I$ a $\D$-form of $V.$
Assume that $M=\oplus_{h \in \C}M_{h}$ is an ordinary $V$-module \cite{DLM1}. A $\D$-form $S=\oplus_{h\in \C}S_{h}$ with $S_{h}=S\cap M_{\h}$ of $M$ over  $I$  is a module for the vertex operator algebra $I$  
and each $S_h$ is a $\D$-form of $M_h.$ 
\end{definition}

We remark that if $I$ is a $\D$-form of $V$, then $I\cap V_{\bar{1}}$ is a $\D$-form of  $V_{\bar{1}}$ over the $\D$-form $I\cap V_{\bar{0}}.$

\section{Virasoro vertex algebras}
\setcounter{equation}{0}
In this section we review  the Virasoro algebra over an integral domain $\mathbb{D}$ and discuss the highest weight modules and related vertex operator algebras. In the rest of the paper we assume that $\F$ is a field with $\ch \F\ne 2.$

Consider the integral domain $\mathbb{D}$ such that $\ha\in \mathbb{D}.$ Set
$Vir_\mathbb{D}=\oplus_{n\in \Z}\mathbb{D}L_n\oplus \mathbb{D}C$
subject to the relation
$$[L_m,L_n]=(m-n)L_{m+n}+\frac{m^3-m}{12}\delta_{m+n,0}C, \  [Vir_\mathbb{D},C]=0$$
for $m,n\in \Z.$ Note that $m^3-m$ is divisible by $3$ for any $m\in \Z,$ the commutators make sense.
In particular we can define the Virasoro algebra $Vir_\F$ for any field $\F$ with $\ch \F\ne 2.$
In the rest of paper we always use $\F$ for such a field.  It is also clear that $Vir_\F=\F\otimes_\D Vir_\D.$

 As in the complex case, for any $c,h\in \mathbb{D}$ we set
$$V(c,h)_\mathbb{D}=U(Vir_\mathbb{D})\otimes_{U(Vir_\mathbb{D}^{\geq0})}\mathbb{D}$$
where $Vir_\mathbb{D}^{\geq0}$ is the subalgebra generated by $L_n$ for $n\geq 0$ and $C,$ and $\mathbb{D}$ is a
 $Vir_\mathbb{D}^{\geq0}$-module such that $L_n1=0$ for $n>0,$ $L_01=h$ and $C1=c.$ Then
  $$V(c,h)_\mathbb{D}=\oplus_{n\geq 0}V(c,h)_\mathbb{D}(n)$$
  is $\Z$-graded where $V(c,h)_\mathbb{D}(n)$ has a basis
  $$\{L_{-n_1}\cdots L_{-n_k}v_{c,h}|n_1\geq \cdots \geq n_k\geq 1, \sum_{i}n_i=n\}$$
  where $v_{c,h}=1\otimes 1.$ The $V(c,h)_\mathbb{D}$ is again is called the Verma module.
  It is easy to see that $L_nV(c,h)_\mathbb{D}(m)\subset V(c,h)_\mathbb{D}(m-n)$ for all $m,n\in\Z.$
Then $V(c,h)_\F$ has a unique maximal graded  submodule $W(c,h)_\F$ such that $L(c,h)_\F=V(c,h)_\F/W(c,h)_\F$ is an irreducible highest weight  $Vir_\F$-module. Moreover, if $\ch\F=0$ then $W(c,h)_\F$ is the unique maximal submodule of $V(c,h)_\F.$ If $\D=\C$,  we will use $V(c,h)$ and $L(c,h)$ to denote $V(c,h)_\C$ and $L(c,h)_\C$, respectively.
   
The following theorem which is the foundation of the framed vertex operator algebra was obtained in \cite{DR1}, \cite{DR2}.
  \begin{thm}\label{harational} Let  $\F$ be an algebraically closed field with $\ch \F\ne 2,7.$

  (1) The $L(\ha, 0)_{\F}$ is a rational vertex operator algebra which has exactly three irreducible modules $L(\ha,h)_\F$ with $h=0,\ha,\se.$

  (2) The fusion product is given by

    $$L(\ha,0)_\BF\times L(\ha,h)_{\BF}=L(\ha,h)_{\BF}$$
    for $h=0,\ha,\se,$

 $$L(\ha,\ha)_\BF\times L(\ha,\ha)_\BF=L(\ha,0)_\BF,$$

  $$L(\ha,\ha)_\BF\times L(\ha,\se)_\BF=L(\ha,\se)_\BF,$$

  $$L(\ha,\se)_\BF\times L(\ha,\se)_\BF=L(\ha,0)_\BF+L(\ha,\ha)_\BF.$$
  \end{thm}

For the discussion of contragredient module later,  we give the following lemma.
\begin{lem} Let $M$ be a sum of highest weight modules for the Virasoro algebra $Vir_\BF.$ Then $\frac{L_1^n}{n!}$ is well defined on $M$ for $n\geq 0.$
\end{lem}
\begin{proof} We can assume that $M$ is a highest weight module with the highest weight $\lambda\in \F$ and the highest weight vector $v.$ Then $M$
is spanned by $L_{-n_1}\cdots L_{-n_k}v$ for $n_1\geq \cdots \geq n_k\geq 1.$ Clearly,  $\frac{L_1^n}{n!}v=0.$ Assume that  $u=L_mw$ and  $\frac{L_1^n}{n!}w$ is well defined. Then $\frac{L_1^n}{n!}u=\sum_{i=0}^n{1-m\choose i}L_{m+i}\frac{L_1^{n-i}}{(n-i)!}w$ is well defined for $n\geq 0.$
\end{proof}

  \section{Modular vertex operator algebra $V(H_\F,\Z+\ha)_{\bar 0}$}
\setcounter{equation}{0}

The modular vertex operator algebra $V(H_\F,\Z+\ha)_{\bar 0}$ is investigated. In particular, a minimal set of generators is determined.

Let $\mathbb{D}$ be an integral domain such that $\ch \D\ne 2$ and  $\frac{1}{2}\in \D.$ .  Let $H_\mathbb{D}=\mathbb{D} a$ be one dimensional free $\mathbb{D}$-module
with a nondegenerate bilinear form $(,)$ determined by $(a,a)=1.$ For short we also set $H=H_\C.$
For $Z=\Z$ or $\Z+\ha$ we let $A(H_\mathbb{D},Z)$ be an associative algebra over $\mathbb{D}$ generated by $a(n)$ for $n\in Z$ subject to the relation $a(m)a(n)+a(n)a(m)=\delta_{m+n,0}$ for $m,n\in Z.$ Also let $A(H_\mathbb{D},Z)_+$ be the subalgebra of $A(H_\mathbb{D},Z)$ generated by $a(n)$ for $n>0$ and  $\mathbb{D}$ is  an $A(H_\mathbb{D},Z)_+$ -module such that $a(n)1=0$ for $n>0.$ Consider the induced module $V(H_\mathbb{D},Z)=A(H_\mathbb{D},Z)\otimes_{A(H_\mathbb{D},Z)_+}\mathbb{D}.$ It is well known that $V(H_\F,\Z+\ha)$ is a simple $A(H_\F,\Z+\ha)$-module and $V(H_\F,\Z)$ is a direct sum of two simple $A(H_\F,\Z)$-modules.

Note that 
$$ 
V(H_\mathbb{D},\Z+\ha)=\wedge[a(-n) | n>0, n\in \Z+\ha]$$ and 
$$V(H_\mathbb{D},\Z)=\wedge[a(-n)|n>0, n\in \Z]+\wedge[a(-n)|n>0, n\in \Z]a(0)$$
 as  $\mathbb{D}$-modules and the action of $A(H_\mathbb{D},Z)$ is as follows:
$a(n)$  acts as $\frac{\partial}{\partial a(-n)}$ for $n>0$ and as the multiplication by $a(n)$ for $n\leq 0.$
The $V(H_\mathbb{D},Z)$ decomposes into $V(H_\mathbb{D},Z)_{\bar 0}\oplus V(H_\mathbb{D},Z)_{\bar 1}$ where $V(H_\mathbb{D},Z)_{\bar s}$ is the submodule spanned by the monomials whose length is congruent to $s$ modulo $2.$
Set $\omega=\frac{1}{2}a(-\frac{3}{2})a(-\frac{1}{2})\in V(H_\mathbb{D},\Z+\ha).$
Let $\tau$ be the canonical automorphism of  $V(H_\mathbb{D}, \Z+\ha)$ so that $\tau|_{V(H_\mathbb{D},\Z)_{\bar s}}=(-1)^s.$

The following result is well known (cf. \cite{FFR,KW, L1, L2, KR}).
 \begin{prop}\label{p4.1} (1) $(V(H_\mathbb{D}, \Z+\ha), Y, 1,\omega)$ is a vertex operator superalgebra generated by $a(-\ha)$ with  $$Y(a(-1/2),z)=\sum_{n\in\Z+\ha}a(n)z^{-n-\ha}$$
 and $V(H_\F, \Z+\ha)$ is a simple vertex operator superalgebra.

(2)  $(V(H_\mathbb{D}, \Z), Y)$ is a $\tau$-twisted  admissible $V(H_\mathbb{D}, \Z+\ha)$-module with $$Y(a(-1/2),z)=\sum_{n\in\Z}a(n)z^{-n-\ha}.$$

(3) $V(H,\Z+\ha)_{\bar 0}$ is isomorphic to $L(\ha,0),$ $V(H ,\Z+\ha)_{\bar 1}$ is isomorphic to $L(\ha,\ha),$
and $V(H,\Z)_{\bar s}$ is isomorphic to $L(\ha,\se)$ for $s=0,1$ as modules for the Virasoro algebra.

\medskip

In (4)-(5), we take  $\D=\Z[\ha].$

(4) $V(H_\D,\Z+\ha)$ is an $\D$-form of $V(H,\Z+\ha)$ and $V(H_\D,\Z)$ is an $\D$-form of $V(H,\Z)$ over $V(H,\Z+\ha).$
Moreover, $V(H_\F,Z)=\F\otimes_\D V(H_\D,Z)$ for any field $\F$ with $\ch\F\ne 2.$

(5) $V(H_\D, \Z+\ha)_{\bar 0}$ is an $\D$-form of  $V(H,\Z+\ha)_{\bar 0}$ and $V(H_\D,Z)_{\bar s}$ is an $\D$-form
of $V(H,Z)_{\bar s}$ over $V(H_\D,\Z+\ha)_{\bar 0}$ for $Z=\Z+\ha,\Z$ and $s=0,1.$

(6) For any field $\F,$ $V(H_\F,\Z+\ha)_{\bar 0}$ is a simple vertex operator algebra and $V(H_\F,Z)_{\bar s}$ for all $s$ are irreducible admissible $V(H_\F,\Z+\ha)_{\bar 0}$-modules.
Moreover, $V(H_\F,\Z)_{\bar 0}$ and $V(H_\F,\Z)_{\bar 1}$ are isomorphic.
\end{prop}
\begin{proof} 

(1)-(5) are clear. We only need to prove (6).  First we show that $V(H_\F,Z)_{\bar s}$ is an admissible $V(H_\F,\Z+\ha)_{\bar 0}$-module. We set
$$V(H_\F,\Z+\ha)_{\bar s}(n)=\<a(-n_1)\cdots a(-n_k)\in V(H_\F,\Z+\ha)|n_1>\cdots >n_k>0,\sum_{i}n_i=n+\ha s\>$$
and
$$V(H_\F,\Z)_{\bar s}(n)=\<a(-n_1)\cdots a(-n_k)\in V(H_\F,\Z)_{\bar s} |n_1>\cdots >n_k>0,\sum_{i}n_i=n\>$$
for $s=0,1.$
Then
$$V(H_\F,Z)_{\bar s}=\oplus_{n\geq 0}V(H_\F,Z)_{\bar s}(n)$$
and
$$u_nV(H_\F,Z)_{\bar s}(m)\subset V(H_\F,Z)_{\bar s}(m+t-n-1)$$
for $u\in (V(H_\F,\Z+\ha)_{\bar 0})_t=V(H_\F,\Z+\ha)_{\bar 0}(t)$ for all $m,n,t.$

We next to prove each $V(H_\F,Z)_{\bar s}$  is an irreducible admissible $V(H_\F,\Z+\ha)_{\bar 0}$-module.
We claim that $V(H_\F,Z)_{\bar s}$ is irreducible under the operators $a(m)a(n)$ for $m,n\in Z.$ We now prove that $V(H_\F,\Z+\ha)_{\bar 0}$ is irreducible under the operators $a(m)a(n)$ for $m,n\in \Z+\ha.$
 That is, any element $0\ne u\in V(H_\F,\Z+\ha)_{\bar 0}$
generates $V(H_\F,\Z+\ha)_{\bar 0}$ using operators $a(m)a(n)$ for $m,n\in \Z+\ha.$ If $u=1$ we see that any $a(-m_1)\cdots a(-m_s)\in V(H_\F,\Z+\ha)_{\bar 0}$ for $m_1>\cdots > m_s>0$ is equal to  $a(-m_1)\cdots a(-m_s)1.$
If $u$ is not a multiple of $1$ there exist $0<m_1<\cdots <m_s$ such that $a(m_1)\cdots a(m_s)u$ is a nonzero multiple of $1.$ That is, $1$ can be generated from $u$ using operators $a(m)a(n).$ As before, any vector $v$ in $V(H_\F,\Z+\ha)_{\bar 0}$ can be generated from $1$ and $u$ using operators $a(m)a(n).$ The proof is similar for the other cases.

 Let $w\in V(H_\F,Z)_{\bar s}.$ 
By Proposition  4.5.7 (which holds in the current situation) of \cite{LL}, we see that
there exist  $c_i\in \F,$ $m_i\in \ha+\Z$ and $n_i\in \Z$  such that 
$$a(m)a(n)w=\sum_{i=1}^kc_i(a(m_i)a(-1/2))_{n_i}w.$$
Note that
$a(m)=a(-\ha)_{m-\ha}$  by (1).   
Since $V(H_\F,Z)_{\bar s}$ is irreducible under the operators $a(m)a(n)$ for $m,n\in \Z+\ha$, 
this shows that $V(H_\F,Z)_{\bar s}$ is an irreducible $V(H_\F,\Z+\ha)_{\bar 0}$-module.  
In particular,
$V(H_\F,\Z+\ha)_{\bar 0}$ is a simple vertex operator algebra.

It remains to show that  $V(H_\F,\Z)_{\bar 0}$ and $V(H_\F,\Z)_{\bar 1}$ are isomorphic. Define
a linear isomorphism $f : V(H_\F,\Z)_{\bar 0}\to V(H_\F,\Z)_{\bar 1}$ such that
$a(-n_1)\cdots a(-n_k)$ for $n_1>\cdots>n_k\geq 0$ is sent to $a(-n_1)\cdots a(-n_k)a(0).$
From the discussion above, it is enough to verify that $ a(s)a(t)f =f a(s)a(t)$ for any $s,t\in\Z$ with $s> t.$ This is clear if both $s,t$ are different from $0.$ We now assume $t=0.$ Then
$$a(s)a(0)f a(-n_1)\cdots a(-n_k)=\frac{\partial}{\partial a(-s)}a(0)a(-n_1)\cdots a(-n_k)a(0)$$
$$f a(s)a(0)a(-n_1)\cdots a(-n_k)=\frac{\partial}{\partial a(-s)}a(0)a(-n_1)\cdots a(-n_k)a(0)$$
are equal for any $n_1>\cdots n_k\geq 0.$ 
\end{proof}

Here we give an explicit expression for the $Y(\omega,z)=\sum_{n\in\Z}L(n)z^{-n-2}.$ For this purpose we define a normal ordering
$$:a(m)a(n):=\left\{\begin{array}{ll} a(m)a(n) & {\rm if}\  m< n\\
0 & {\rm if}\  m=n\\
-a(n)a(m)  & {\rm if}\  m> n
\end{array}\right.$$
It is easy to check that
\begin{equation*}
L(n)=\ha\sum_{j\in Z}j:a(-j)a(n+j):
\end{equation*}
if $Z=\Z+\ha$ or $n\ne 0,$ and
\begin{equation*}
L(0)=\frac{1}{16}+\ha\sum_{j\in \Z}j:a(-j)a(j):.
\end{equation*}

The following result will be useful later.
\begin{lem}\label{al1} For $t\geq 1$ let $u^t=a(-t-\ha)a(-\ha),$ $v=a(-\frac{5}{2})a(-\frac{3}{2})$ and $J=u^3-3v.$ Then
$$[u^t_n,a(m-\ha)]=\left((-1)^t{m+n-1\choose t}-{m-1\choose t}\right)a(m+n-t-\ha),$$
$$[v_n,a(m-\ha)]=\frac{(m-1)(m+n-3)(2m+n-4)}{2}a(m+n-3-\ha)$$
and
$$[J_n,a(m-\ha)]=\lambda_{m,n}a(m+n-3-\ha)$$
for $m,n\in \Z$ where
$$\lambda_{m,n}=-{m+n-1\choose 3}-{m-1\choose 3}-\frac{3(m-1)(m+n-3)(2m+n-4)}{2}.$$
\end{lem}
\begin{proof}  Note that $a(-\ha)_{m-1}=a(m-\ha),$ $a(-\ha)_tu^t=a(-\ha),$ $a(-\ha)_0u^t=-a(-t-\ha),$
$a(-\ha)_iu^t=0$ for $i\geq 0$ with $i\ne 0,t.$ Also, $a(-t-\ha)_s=(-1)^t{s\choose t}a(-\ha)_{s-t}$
We have
\begin{eqnarray*}
&& [u_n^t, a(m-\frac{1}{2})]=-[a(-\frac{1}{2})_{m-1}, u_n^t]\\
& &\ \ \ \ =-\sum_{i\geq 0}{m-1\choose i}(a(-\frac{1}{2})_iu^t)_{m+n-1-i}\\
& &\ \ \ \ =a(-t-\ha)_{m+n-1}-{m-1\choose t}a(-\ha)_{m+n-t-1}\\
& & \ \ \ \ =[(-1)^t{m+n-1\choose t}-{m-1\choose t}]a(m+n-t-\ha).
\end{eqnarray*}
The identity $[v_n,a(m-\ha)]=\frac{(m-1)(m+n-3)(2m+n-4)}{2}a(m+n-3-\ha)$ can be proved similarly.
\end{proof}

Recall from \cite{JLM} that $Vir_\F$ is a restricted Lie algebra whose $p$-map is given by
$C^{[p]}=C$ and for $n\in \Z$,
$$
(L_{n})^{[p]}=\begin{cases}L_{np}& \   \   \mbox{if }p\mid n\\
0& \   \   \mbox{if }p\nmid n.
\end{cases}
$$
Set $J(n)=J_{3+n}.$ Similarly 
we define 
$$
J(n)^{[p]}=\begin{cases}J(np)& \   \   \mbox{if }p\mid n\\
0& \   \   \mbox{if }p\nmid n.
\end{cases}
$$
Recall that $V(H_\F,\ha+\Z)$ and $V(H_\F,\Z)$ are $Vir_\F$-modules.

\begin{prop}\label{pcenter1} For any $n\in\Z,$ $L_n^p-(L_{n})^{[p]}=0,$   $J(n)^p-J(n)^{[p]}=0,$   on $V(H_\F,\ha+\Z), V(H_\F,\Z).$ In particular, both $V(H_\F,\ha+\Z)$ and $V(H_\F,\Z)$
are modules for the restricted Virasoro algebra.
\end{prop}

\begin{proof} We only prove the result for $V(H_\F,\ha+\Z)$ and the proof for $V(H_\F,\Z)$ is similar.  We first show that $[L_n^p-(L_{n})^{[p]}, a(m)]=0$ for $m\in \ha+\Z.$
It is easy to see that
$$(ad L_n)^pa(m)=-(m+n/2)(m+n+n/2)\cdots (m+(p-1)n+n/2)a(m+pn).$$
Clearly, $(ad L_n)^pa(m)=0$ if $p\nmid n,$ and
$$(ad L_n)^pa(m)=-m^pa(m+pn)=-ma(m+pn)=[L_{pn},a(m)]$$
if $p\mid n.$ 
It follows from Lemma in Section 21.4 of \cite{H}  that
$$[L_n^p,a(m)]=(ad L_n)^pa(m)$$
and consequently,
$$[L_n^p-(L_{n})^{[p]}, a(m)]=0.$$

Since $V(H_\F,\ha+\Z)$ is an irreducible module of  $A(H_\mathbb{\F},\ha+\Z)$-module,and $(L_n^p-(L_{n})^{[p]})$ acts on $V(H_\F,\ha+\Z)$ as a constant. If $n=0$ then $(L_n^p-(L_{n})^{[p]})\1=0,$ and
if $n\ne 0$ then $(L_0^p-(L_{0})^{[p]})\1$ lies in the intersection of $V(H_\F,\ha+\Z)(0)$ and $V(H_\F,\ha+\Z)(-pn).$ This forces $(L_n^p-(L_{n})^{[p]})\1=0.$ 

Using Lemma \ref{al1} we can prove  $J(n)^p-J(n)^{[p]}=0$ in the same fashion.  The proof is complete. \end{proof}

\section{Generators of $V(H_\F,\Z+\ha)_{\bar 0}$ }

In this section we prove that $V(H_\F,\Z+\ha)_{\bar 0}$ is generated by $\omega$ and $J.$ 

\begin{thm}\label{tm1} If $\ch \F\ne 7,$ then $V(H_\F,\ha+\Z)_{\bar 0}$ is isomorphic $L(\ha,0)_\F,$
$V(H_\F,\ha+\Z)_{\bar 1}$ is isomorphic $L(\ha,\ha)_\F,$ and $V(H_\F,\Z)_{\bar s}$ is isomorphic $L(\ha,\se)_\F$
for $s=0,1$ as modules for the Virasoro algebra $Vir_\F.$ If $\ch \F= 7,$ $V(H_\F,\ha+\Z)_{\bar 0}$  is generated by $\omega$ and either $v$ or $J.$ 
\end{thm}
\begin{proof} This result when $\ch \F\ne 7$ was given in \cite{DR2}. But there was a gap in the proof.  Our proof here works for all $\F$ with $\ch\F\ne 2.$ Of course, the proof present here also works for $\F=\C.$

Let $W$ be the subalgebra of $V(H_\F,\Z+\ha)_{\bar 0}$ generated by $\omega$ if $\ch \F\ne 7,$ and by $\omega, v$ if $\ch \F= 7.$ We first prove that $V(H_\F,\ha+\Z)_{\bar 0}=W.$
From the proof of (6) of Proposition \ref{p4.1}, 
we need to argue that $a(-m-\ha)a(-n-\ha)\in W$ for $m,n\geq 0.$ 
 We prove this by induction on
$s=m+n.$  First we have the well known relation
$$[L(p),a(q)]=-(q+\frac{p}{2})a(p+q)$$
for $p\in\Z$ and $q\in\frac{1}{2}+\Z.$

If $s=1$ then $a(-3/2)a(-1/2)=2L(-2)\1=2\omega\in W.$ If $s=2$ then $a(-5/2)a(-1/2)=L(-1)\omega\in W.$ If $s=3$
consider the linear system
$$L(-1)a(-\frac{5}{2})a(-\ha)=3a(-\frac{7}{2})a(-\ha)+a(-\frac{5}{2})a(-\frac{3}{2})$$
$$L(-2)a(-\frac{3}{2})a(-\ha)=(2+\ha)a(-\frac{7}{2})a(-\ha)-\frac{3}{2}a(-\frac{5}{2})a(-\frac{3}{2}).$$
Since the coefficient matrix
$$\left(\begin{array}{cc} 3 & 1\\ 2+\ha & -\frac{3}{2}\end{array}\right)$$
has determinant
$-7$,  we see that both $a(-\frac{7}{2})a(-\ha)$ and $a(-\frac{5}{2})a(-\frac{3}{2})$ are in $W$ if $\ch\F\ne 7.$ If $\ch\F= 7,$ then $a(-\frac{7}{2})a(-\ha)$ lies in $W$ as 
$v=a(-\frac{5}{2})a(-\frac{3}{2})$ is 
assumed to be in $W.$

We now assume that the result holds for $m+n\leq s$  with $s\geq 4.$
We have linear equations
\[
\begin{split}
L(-1)a(-s-\ha)a(-\ha)=(s+1)a(-s-\frac{3}{2})a(-\ha)+a(-s-\ha)a(-\frac{3}{2})
\qquad \quad  (1) \\
v_2a(-s-\ha)a(-\ha)=-(s+1)^3a(-s-\frac{3}{2})a(-\ha)-a(-s-\ha)a(-\frac{3}{2}) 
\qquad \quad  (2) \\ 
L(-1)a(-s+1-\ha)a(-\frac{3}{2})=sa(-s-\frac{1}{2})a(-\frac{3}{2})+2a(-s+1-\ha)a(-\frac{5}{2})   \qquad \quad  (3)  \\
L(-2)a(-s+\ha)a(-\ha)=(s+\ha)a(-s-\frac{3}{2})a(-\ha)+\frac{3}{2}a(-s+\ha)a(-\frac{5}{2})  \qquad  \quad (4)
\end{split}
\]
Note that if $s\not\equiv 0, -1,-2$,  the first two equations have unique solutions. That is, both $a(-s-\frac{3}{2})a(-\ha)$ and $a(-s-\ha)a(-\frac{3}{2})$ are linear combinations
of $L(-1)a(-s-\ha)a(-\ha)$ and $v_2a(-s-\ha)a(-\ha).$ If $s\equiv 0, -1,-2,$ the coefficient matrix of the linear system consisting of equations (1), (3), (4)
$$\left(\begin{array}{ccc} s+1 & 1& 0 \\ 0&  s & 2\\ s+\ha & 0& \frac{3}{2}\end{array}\right)$$
has determinant $\frac{3}{2}s(s+1)+2(s+\ha)$, 
which is $\ne 0$  if $s\equiv 0, -1.$
By induction assumption, the vectors in the left hand side of the linear system lie in $W.$ Solving the linear system above shows that $a(-s-1-\ha)a(-\ha)$ lies in $W$  with 
$s \not\equiv -2.$

Let  $m,n\geq 0$ such that $m+n=s+1.$ By Proposition 4.5.7  of \cite{LL},  we see that
\begin{eqnarray*}
& &a(-m-\ha)a(-n-\ha)=a(-\ha)_{-m-1}a(-\ha)_{-n-1}\1\label{e4.3}\\
& &\ \ \ =\sum_{i=0}^{n}{-m-1\choose i}(a(-\ha)_{-m-1-i}a(-\ha))_{-n-1+i}\1. \nonumber
\end{eqnarray*}
Since $a(-\ha)_{-m-1-i}=a(-m-i-\ha)$ and $i\leq n,$ this 
implies that $a(-m-\ha)a(-n-\ha)$ lies in $W$ for all $m+n\leq s+1$  with $s \not\equiv -2$ from induction assumption.

The proof for $s\equiv -2$ is much more complicated and is divided into several steps. From the discussion above, it suffices to show that $a(-s-1-\ha)a(-\ha)\in W.$

(1) Let $s=2np-2.$ Then
\[
\begin{split}
&\ \ \ \ \ L(-1)a(-s+i-\ha)a(-i-\ha)\\
&=-(i+1)a(-s-1+i-\ha)a(-i-\ha)+(i+1)a(-s+i-\ha)a(-i-1-\ha)
\end{split}
\]
for $0\leq i\leq s.$ Thus if $i\leq p-2$
$$a(-s+i-\ha)a(-i-1-\ha)\equiv a(-s-1-\ha)a(-\ha)\ (\mod\  W).$$
Similarly,
$$a(-s+jp+i-\ha))a(-jp-i-1-\ha)\equiv a(-s-1+jp-\ha)a(-jp-\ha)\ (\mod\  W)$$
for $j<2n$ and $i=0,...,p-2.$

(2)  Lemma \ref{al1} gives
\begin{eqnarray*}
& &u^t_0a(-i-\ha)a(-j-\ha)=[(-1)^t{-i-1\choose t}-{-i-1\choose t}]a(-i-t-\ha)a(-j-\ha)\\
& &\ \ \ \ \ \ +[(-1)^t{-j-1\choose t}-{-j-1\choose t}]a(-i-\ha)a(-j-t-\ha).
\end{eqnarray*}
In particular, for any positive odd integer $n,$
\begin{eqnarray*}
& &u^{np-2}_0a(-i-\ha)a(-j-\ha)=-2{-i-1\choose np-2}a(-i-np+2-\ha)a(-j-\ha)\\
& &\ \ \  -2{-j-1\choose np-2}a(-i-\ha)a(-j-np+2-\ha).
\end{eqnarray*}

(3) Let $s=np-2$ with $n$ odd.  Then $ a(-np+2-\ha)a(-\ha)$ lies in $W$ by induction assumption.
 Note that
$$L(-1)a(-np+2-\ha)a(-\ha)=-a(-np+1-\ha)a(-\ha)+a(-np+2-\ha)a(-1-\ha).$$
Also,
\begin{eqnarray*}
& &u^{np-2}_0a(-\ha)a(-1-\ha)=-2{-1\choose np-2}a(-np+2-\ha)a(-1-\ha)\\
& &\ \ \ -2{-2\choose np-2}a(-\ha)a(-np+1-\ha)\\
& &\ \ \ =2a(-np+2-\ha)a(-1-\ha)+2a(-np+1-\ha)a(-\ha).
\end{eqnarray*}
Since $L(-1)a(-np+2-\ha)(-\ha)$ and $u^{p-2}_0a(-\ha)a(-1-\ha)$ lie in $W$,  we immediately see that
both $a(-np+1-\ha)a(-\ha)$ and $a(-np+2-\ha)a(-1-\ha)$ lie in $W.$

(4) Let $s=np-2$ with $n$ even. Then
\begin{eqnarray*}
& &\ \ \  u^{(n-1)p-1}_{-1}a(-p+1-\ha)a(-\ha)\\
& &=[{-p-1\choose (n-1)p-1}-{-p\choose (n-1)p-1}]a(-np+1-\ha)a(-\ha)\\
& &\ \ \ +[{-2\choose (n-1)p-1}-{-1\choose (n-1)p-1}]a(-p+1-\ha)]a(-(n-1)p-\ha)\\
& &=(n-1)a(-np+1-\ha)a(-\ha)+a(-(n-1)p-\ha)a(-p+1-\ha)\\
& &\equiv na(-np+1-\ha)a(-\ha)\ (\mod\ W),
\end{eqnarray*}
where we have used (1) in the last step. So if $n$ is not a multiple of $p,$ $a(-np+1-\ha)a(-\ha)$ lies in $W.$

Now we assume that $n$ is a multiple of $p.$ Write  $np=mp^k$ with $m$ being even, and $m,p$ being coprime. Clearly,
$k\geq 2.$ We have
\begin{eqnarray*}
& &\ \ \  u^{(m-1)p^k-1}_{-1}a(-p^k+1-\ha)a(-\ha)\\
& &=[{-p^k-1\choose (m-1)p^k-1}-{-p^k\choose (m-1)p^k-1}]a(-mp^k+1-\ha)a(-\ha)\\
& &\ \ \ +[{-2\choose (m-1)p^k-1}-{-1\choose (m-1)p^k-1}]a(-p^k+1-\ha)a(-(m-1)p^k-\ha).
\end{eqnarray*}
A straightforward computation gives 
$${-p^k-1\choose (m-1)p^k-1}=\frac{\prod_{i=0}^{p^k-1}[{(m-1)p^k+i}]}{p^k!}=\frac{\prod_{i=0}^{p^{k-1}-1}[(m-1)p^k+ip]}{\prod_{i=1}^{p^{k}-1}ip}=m-1,$$
$${-p^k\choose (m-1)p^k-1}=\frac{\prod_{i=0}^{p^k-2}[{(m-1)p^k+i}]}{(p^k-1)!}=\frac{\prod_{i=0}^{p^{k-1}-2}[(m-1)p^k+ip]}{\prod_{i=1}^{p^{k-1}-1}ip}=0,$$
$${-2\choose (m-1)p^k-1}=\frac{[(m-1)p^k]!}{[(m-1)p^k-1)]!}=0,\  {-1\choose (m-1)p^k-1}=1.$$
Note that 
$$a(-p^k+1-\ha)a(-(m-1)p^k-\ha)\equiv {-p^k+1\choose (m-1)p^k}a(-mp^k+1-\ha)a(-\ha) \ (\mod\ W)$$
and 
$${-p^k+1\choose (m-1)p^k}=-\frac{\prod_{i=1}^{p^k-2}[{(m-1)p^k+i}]}{(p^k-2)!}=-1.$$
Thus 
$$ u^{(m-1)p^k-1}_{-1}a(-p^k+1-\ha)a(-\ha)\equiv ma(-mp^k+1-\ha)a(-\ha) \ (\mod \ W)$$
and $a(-mp^k+1-\ha)a(-\ha)$ lies in $W.$  

So for any $s\geq 1$ we know that $a(-s-1-\ha)a(-\ha)\in W.$ As a result, $V(H_\F,\ha+\Z)_{\bar 0}=W.$ The rest of the theorem is clear by noting that if $\ch \F\ne 7$, 
$ u^t$ can be generated from $\omega.$
\end{proof}

The representation theory for $V(H_\F,\ha+\Z)_{\bar 0}\cong L(\ha,0)_\F$ 
when $\ch\F\ne 7$ is clear from \cite{DR2}. We will classify irreducible  $V(H_\F,\ha+\Z)_{\bar 0}$-modules  with $\ch\F=7$ in the next two sections.

\section{A spanning set of $V(H_\F,\ha+\Z)_{\bar 0}$}
We assume that $\ch\F=7$ and give a nice spanning set of  $V(H_\F,\ha+\Z)_{\bar 0}$ in this section. 
By Theorem \ref{tm1},  $V(H_\F,\ha+\Z)_{\bar 0}$ is generated by $\omega, J.$ Moreover $J$ is a highest weight 
vector for the Virasoro algebra with highest weight $4$ \cite{DR2}. 

\begin{thm}\label{tspan} $V(H_\F,\ha+\Z)_{\bar 0}$ has a spanning set
$$L(-n_1)\cdots L(-n_k)J_{-m_1}\cdots J_{-m_l}\1$$
for $n_1\geq \cdots\geq n_k\geq 2,$ $m_1\geq \cdots\geq m_l\geq 1.$
\end{thm}
\begin{proof} The proof is divided into several steps.

1) Let $i\geq -1.$ Then $J_iJ$ lie in $U(Vir_\F)\1+U(Vir_\F)J.$ Note that $J_iJ\in V(H_\F,\ha+\Z)_{\bar 0}(8-i-1).$ So it is good enough to show that $V(H_\F,\ha+\Z)_{\bar 0}(8-i-1)$ is a subspace
of $U(Vir_\F)\1+U(Vir_\F)J.$ This is clear for $i\geq 3.$ We deal with other $i$ case by case.

(1a) Since $V(H_\F,\ha+\Z)_{\bar 0}(5)$ is spanned by $a(-4-\ha)a(-\ha), a(-3-\ha)a(-1-\ha)$ and
$$L(-1)a(-3-\ha)a(-\ha)=4a(-4-\ha)a(-\ha)+a(-3-\ha)a(-1-\ha),$$
$$ L(-1)a(-2-\ha)a(-1-\ha)=3a(-3-\ha)a(-1-\ha),$$
 $V(H_\F,\ha+\Z)_{\bar 0}(5)$ is a subspace of   $U(Vir_\F)\1+U(Vir_\F)J.$

(1b) The subspace 
$V(H_\F,\ha+\Z)_{\bar 0}(6)$ is spanned by $a(-5-\ha)a(-\ha), a(-4-\ha)a(-1-\ha), a(-3-\ha)a(-2-\ha).$ The linear system
$$L(-1)a(-4-\ha)a(-\ha)=5a(-5-\ha)a(-\ha)+a(-4-\ha)a(-1-\ha)$$
$$L(-1)a(-3-\ha)a(-1-\ha)=4a(-4-\ha)a(-1-\ha)+2a(-3-\ha)a(-2-\ha)$$
$$L(-2)a(-3-\ha)a(-\ha)=a(-5-\ha)a(-\ha)+5a(-3-\ha)a(-2-\ha)$$
has a unique solution for $a(-5-\ha)a(-\ha), a(-4-\ha)a(-1-\ha), a(-3-\ha)a(-2-\ha)$ in terms of $L(-1)a(-4-\ha)a(-\ha), L(-1)a(-3-\ha)a(-1-\ha), L(-2)a(-3-\ha)a(-\ha).$ Thus
$V(H_\F,\ha+\Z)_{\bar 0}(6)$ is a subspace of   $U(Vir_\F)\1+ U(Vir_\F)J.$

 (1c) The subspace $V(H_\F,\ha+\Z)_{\bar 0}(7)$ is spanned by $a(-6-\ha)a(-\ha), a(-5-\ha)a(-1-\ha), a(-4-\ha)a(-2-\ha).$ Solving the linear system
 $$L(-1)a(-5-\ha)a(-\ha)=6a(-6-\ha)a(-\ha)+a(-{5}-\ha)a(-1-\ha)$$
 $$L(-1)a(-4-\ha)a(-1-\ha)=5a(-5-\ha)a(-1-\ha)+2a(-4-\ha)a(-2-\ha)$$
$$L(-1)a(-3-\ha)a(-2-\ha)=4a(-{4}-\ha)a(-2-\ha)$$
again shows that $V(H_\F,\ha+\Z)_{\bar 0}(7)$ is a subspace of  $U(Vir_\F)\1+U(Vir_\F)J.$

(1d) The subspace 
$V(H_\F,\ha+\Z)_{\bar 0}(8)$ is spanned by $a(-7-\ha)a(-\ha), a(-6-\ha)a(-1-\ha), a(-5-\ha)a(-2-\ha), a(-4-\ha)a(-3-\ha).$ Solving the linear system
$$L(-1)a(-6-\ha)a(-\ha)=a(-6-\ha)a(-1-\ha)$$
 $$L(-1)a(-5-\ha)a(-1-\ha)=6a(-6-\ha)a(-1-\ha)+2a(-5-\ha)a(-2-\ha)$$
 $$L(-1)a(-4-\ha)a(-2-\ha)=5a(-5-\ha)a(-1-\ha)+3a(-4-\ha)a(-3-\ha)$$
 $$L(-2)a(-5-\ha)a(-\ha)=3a(-7-\ha)a(-\ha)+5a(-5-\ha)a(-2-\ha)$$
 again gives the desired result.

One can also give a conceptional proof somehow by noting from \cite{DR2} that $V(H_\F,\ha+\Z)_{\bar 0}$ is a completely reducible $Vir_\F$-module and $L(\ha,0)_\F$ has exactly 2 inequivalent irreducible
modules $L(\ha,0)_\F$ and $L(\ha,\ha)_\F.$

2) For $m,n\in\Z,$
$$[J_m,J_n] = \sum_{i=0}^\infty\binom{m}{i}(J_iJ)_{m+n-i}.$$
We need to compute $J_iJ$ explicitly for $8\geq i\geq 0$ as $J_iJ=0$ if $i>8.$  A straightforward computation using Lemma  \ref{al1} gives
$$J_0J=a(-6-\ha)a(-\ha)+2a(-5-\ha)a(-1-\ha)+2a(-4-\ha)a(-2-\ha),$$
$$J_1J=-a(-5-\ha)a(-\ha)-a(-3-\ha)a(-2-\ha),$$
$$J_2J=-a(-4-\ha)a(-\ha)+a(-3-\ha)a(-1-\ha),$$
$$J_3J=a(-3-\ha)a(-\ha)-a(-2-\ha)a(-1-\ha)=-J+3L(-1)^2\omega$$
$$J_4J=3a(-2-\ha)a(-\ha)=3L(-1)\omega,$$
$$J_5J=3a(-1-\ha)a(-\ha)=6\omega,$$
$$J_6J=J_7J=0.$$

By 1), all $J_iJ$ are linear combinations of
\[
L(-m_1)\cdots L(-m_s)\1,\quad
L(-n_1)\cdots L(-n_t)J
\]
where $m_1\geq m_2\geq\dots\geq m_s\geq 2$, $n_1\geq n_2\geq \dots\geq n_t\geq 1$
and $s,t\leq 3$. Note that for any vertex operator algebra $V,$ $a,b\in V$
and $m,n\in\Z,$ $(a_mb)_n$ is a linear combination of operators
$a_sb_t$ and $b_ta_s$ for $s,t\in \Z.$  Using
\begin{equation*}\label{lj}
[L(m),J_n]=(-m-n-1)J_{m+n}+4(m+1)J_{m+n}=(3m-n+3)J_{m+n}
\end{equation*}
we see for any $m,n\in\Z$,  the commutator $[J_m,J_n]$ is a linear combinations of
\[
L(p_1)\cdots L(p_s),\quad
L(q_1)\cdots L(q_t)J_r
\]
where $p_1,\dots,p_s,q_1,\dots,q_t,r\in\Z$ and $s,t\leq 3$.

3) By Theorem \ref{tm1} and 2) we know that $V(H_\F,\ha+\Z)_{\bar 0}$  is spanned by
\begin{align*}
&\left\{
L(m_1)L(m_2)\cdots L(m_s)J_{n_1}J_{n_2}\cdots J_{n_t}\1\,|\,
m_i,n_j\in\Z
\right\}.
\end{align*}
The exact proof of Proposition 3.4 of \cite{DN} shows
that
$V(H_\F,\ha+\Z)_{\bar 0}$ is spanned
by
$$L(-n_1)\cdots L(-n_k)J_{-m_1}\cdots J_{-m_l}\1$$
for $n_1\geq \cdots n_k\geq 2,$ $m_1\geq \cdots m_l\geq 1.$
\end{proof}

\section{Classification of irreducible modules}
In this section we classify the irreducible $V(H_\F,\ha+\Z)_{\bar 0}$-modules with $\ch \F=7.$ 
The main tool is the $A(V)$-theory developed in \cite{Z} and \cite{DR1}.

\begin{lem}\label{3modules}  Assume $p=7.$ Then 
$$o(J)|_{V(H_\F,\ha+\Z)_{\bar 0}(0)}=0,o(J)|_{V(H_\F,\ha+\Z)_{\bar 1}(0)}=1, o(J)|_{V(H_\F,\Z)_{\bar i}(0)}=4$$ 
for $i=0,1,$ 
where $o(J)$ denotes the degree preserving operator of $J$.
In particular, the irreducible $V(H_\F,\ha+\Z)_{\bar 0}$-modules $V(H_\F,\ha+\Z)_{\bar 0}, V(H_\F,\ha+\Z)_{\bar 1}, V(H_\F, \Z)_{\bar 0}$ are inequivalent.
\end{lem}
\begin{proof} It is clear that  $o(J)|_{V(H_\F,\ha+\Z)_{\bar 0}(0)}=0.$ Note that $V(H_\F,\ha+\Z)_{\bar 1}(0)=\F a(-\ha),$ $V(H_\F,\Z)_{\bar 0}(0)=\F 1$ and $V(H_\F,\Z)_{\bar 1}(0)=\F a(0).$
From Lemma \ref{al1} we have
$$o(J)a(-\ha)=(u^3_3-3v_3)a(-\ha)=-[{2\choose 3}+{-1\choose 3}]a(-\ha)=a(-\ha).$$

The calculation of $o(J)$ on $1,a(0)\in V(H_\F,\Z)$ are more complicated as the definition of $Y(J,z)$ involves more.
Set
$$\Delta(z)=\frac{1}{4}\sum_{m,n\geq 0}C_{mn}a(m+\ha)a(n+\ha)z^{-m-n-1}$$
where
$$C_{mn}=\frac{m-n}{m+n+1}{-\ha\choose m}{-\ha\choose n}.$$
One can easily check that $C_{mn}$ is a well defined number in $\Z_p$ even $m+n+1\equiv 0.$  For $n_1,...,n_r\in \Z$ we define normal order
$$\mbox{$\circ\atop\circ$}a(n_1)\cdots a(n_r)\mbox{$\circ\atop\circ$}=(-1)^{|\sigma|}a(n_{\sigma 1})\cdots a(n_{\sigma r})$$
where $\sigma\in S_r$ such that $n_{\sigma 1}\leq \cdots \leq n_{\sigma r}.$ Recall $a(z)=\sum_{n\in\Z}a(n)z^{-n-1/2}$ on  $V(H_\F,\Z).$ For $n\geq 0$ set $\partial_n=\frac{1}{n!}(\frac{d}{dz})^n.$
Let $u=a(-n_1-\ha)\cdots a(-n_r-\ha)\in V(H_\F,\ha+\Z)$ and define
$$\bar{Y}(a(-n_1-\ha)\cdots a(-n_r-\ha),z)=\mbox{$\circ\atop\circ$}(\partial_{n_1}a(z))\cdots (\partial_{n_r}a(z))\mbox{$\circ\atop\circ$}.$$
From \cite{FFR} we know that
$$Y(u,z)=\bar{Y}(e^{\Delta(z)}u,z)$$
on $V(H_\F,\Z)$ for $u\in V(H_\F,\ha+\Z).$

Since $C_{03}-C_{30}=2$ and $C_{12}-C_{21}=-1,$  we compute
$$\Delta(z)J=(\frac{1}{4}(C_{03}-C_{30})-\frac{3}{4}(C_{12}-C_{21}))z^{-4}=3z^{-4}.$$
Thus
$$o(J)1=3+{-1/2\choose 3}\mbox{$\circ\atop\circ$} a(0)^2 \mbox{$\circ\atop\circ$} -3{-1/2 \choose 2}(-\ha)\mbox{$\circ\atop\circ$}a(0)^2\mbox{$\circ\atop\circ$}=3,\ \ o(J)a(0)=3a(0),$$
as desired. \end{proof}

For short, let $V=V(H_\F,\ha+\Z)_{\bar 0}.$
Next we use  $A(V)$-theory developed in \cite{Z} and \cite{DR1} to prove that the irreducible $V$-modules given in Lemma \ref{3modules}  give a complete list of inequivalent irreducible $V$-modules.    For $a\in V$, we set $[a]=a+O(V).$  
The exact same proof of Theorem 3.5 of \cite{DN} gives
\begin{thm}\label{t5.1} The algebra $A(V)$ is spanned by
$\mathcal{S} = \{[\omega]^{*s}[J]^{*t}\,|\, s,t\geq 0\}$
where we omit the product sign $*$ for simplicity.
\end{thm}

We next discuss the various relations among $[\omega]^s,$ $[J]^t$ for $s,t\geq 1.$
\begin{lem}\label{l5.2} We have the following relations

(1) $[\omega]([\omega]-4)=0,$

(2) $[\omega][J]=4[J],$

(3) $[J]^2=[\omega]-3[J],$

(4) $[J]([J]-1)([J]-3)=0.$
\end{lem}
\begin{proof} (1) From the proof of Theorem 5.2 of \cite{DR2},  we know that $V$ is completely reducible module for $Vir_\F$ or $L(\ha,0)_\F.$ So $U(Vir_\F)\1=L(\ha,0)_\F$ and $O(L(\ha,0)_\F)$ is a subspace of $O(V).$
Also ${\bf u}=L(-2)^2-2L(-4))\1=0$ in $L(\ha,0)_\F$ by Lemma 5.1 of \cite{DR2}. Then
$$[\omega]^2=[L(-2)\omega]-2[\omega]=2[L(-4)\1]-2[\omega]=[L(-1)^2\omega]-2[\omega]=4[\omega],$$
equivalently
$$[\omega]([\omega]-4)=0.$$

(2) Using  (1b) in the proof of Theorem \ref{tspan},  we have
\begin{eqnarray*}
& &L(-2)J=L(-2)(a(-3-\ha)a(-\ha)-3a(-2-\ha)a(-1-\ha))\\
& &\ \ \ \ =a(-5-\ha)a(-\ha)+2a(-3-\ha)a(-2-\ha)\\
& &\ \ \ =3L(-1)a(-4-\ha)a(-\ha)+L(-1)a(-3-\ha)a(-1-\ha)
\end{eqnarray*}
which is equal to $-a(-4-\ha)a(-\ha)+2a(-3-\ha)a(-1-\ha)$ modulo $O(V).$ By (1a) in the proof of Theorem \ref{tspan},
$$-a(-4-\ha)a(-\ha)+2a(-3-\ha)a(-1-\ha)=(-2)L(-1)J$$
which is equal to $J$ modulo $O(V).$ Thus
$$[\omega][J]=[L(-2)J]-4[J]=-3[J]=4[J].$$

(3) From Theorem \ref{tspan}, $[J]^2$ is a linear combination of $[\omega]^i$ with $i=0,1,2,3, 4$ and $[\omega]^j[J]$ with $j=0,1,2.$ Using 1) and 2),  we see that
there are constants $a,b,c$ such that
$$[J]^2=a[\omega]+b[J]+c.$$
Applying this identity to $\1$ and noting that $o(\omega)\1=o(J)\1=0$ assert $c=0.$  Applying the identity to $a(-\ha)$ and $1\in V(H_\F,\Z)$ yields the linear system
$$1=4a+b, \ \ 2=4a+3b$$
which has a unique solution $a=1, b=4.$

(4) Using (2) and (3),  we have  $$[J]([J]-1)([J]-3)=[J]([J]^2-4[J]+3)=[J]([\omega]+3)=0,$$
as desired.
\end{proof}

By Theorem \ref{t5.1} and Lemma \ref{l5.2}, we immediately have
\begin{thm}\label{t5.2} The algebra $A(V)$ is isomorphic to $\F[x]/(x(x-1)(x-3))$ where the isomorphism is given by
$[J]^t\mapsto x^t$ for $t\geq 0.$ In particular, $A(V)$ is semisimple and  $V$ has exactly 3 inequivalent irreducible modules.
\end{thm}

The one to one correspondence 
between irreducible $V$-modules and irreducible $A(V)$-modules given in \cite{DR1}  gives
\begin{thm} For $p>2,$ $V(H_\F,\ha+\Z)_{\bar 0}$ has exactly 3 inequivalent irreducible admissible modules $V(H_\F,\ha+\Z)_{\bar i},$ $V(H_\F,\Z)_{\bar 0}$ for $i=0,1.$ Moreover, $V(H_\F,\Z)_{\bar 0}$ and $V(H_\F,\Z)_{\bar 1}$ are isomorphic $V(H_\F,\ha+\Z)_{\bar 0}$-modules.
\end{thm}

We conjecture  that $V(H_\F,\Z+\ha)_{\bar 0}$ is rational and the fusion rules are the same as before \cite{DR1} if $\ch \F=7.$   A positive answer to this conjecture will extend the theory of modular framed vertex operator algebra to any field
$\F$ with $\ch\F\ne 2.$  Determining the decomposition of $V(H_\F,Z)$ in to a direct sum of irreducible 
$L(\ha,0)_\F$-modules is another problem. According to \cite{DR2}, $V(H_\F,Z)$ is a completely reducible $L(\ha,0)_\F$-module. We hope the decomposition will lead to character formulas for 
$L(\ha,0)_\F$ and $L(\ha,\ha)_\F.$ A related problem is to determine the singular vectors of the Verma modules $V(\ha,0)_\F$ and $V(\ha,\ha)_\F.$ The desired  decompositions of  $L(\ha,0)_\F$-modules $V(H_\F,Z)$ should give us the degrees of the singular vectors. But it is better to find explicit expressions of singular vectors for the purpose of studying general modular vertex operator algebras $L(c_m,0)_\F$ for $m\geq 4$ where $c_m=1-\frac{6}{m(m+1)}.$

\section{Framed vertex operator algebras}
\setcounter{equation}{0}
In the last half of this paper,  we will define and study the modular vertex operator algebras based on the representation theory of $L(\ha,0)_\F$ in \cite{DR1,DR2}. We assume that $\F$ is an algebraically closed field and $\ch\, \F\neq 2,7$. 

 A simple vertex operator superalgebra
$V=\oplus_{n\geq 0}V_n$ over $\F$ is called a {\em framed vertex operator superalgebra\/} (FVOSA) if  there exist $\omega_i\in V$ for $i=1$,~$\ldots$,~$r$ such that
(i) each $\o_i$ generates a copy of the simple Virasoro vertex operator
algebra of central charge $\frac{1}{2}$ and
the component operators $L^i(n)$ of
$Y(\omega_i,z)=\sum_{n\in\Z}L^i(n)z^{-n-2}$
satisfy $[L^i(m),L^i(n)]=(m-n)L^i(m+n)+\frac{m^3-m}{24}\delta_{m+n,0};$
(ii) the $r$ Virasoro algebras are mutually commutative;
and (iii) $\omega=\omega_1+\cdots+\omega_{r}$.
The set $\{\o_1,\ldots,\o_r\}$ is called a {\em Virasoro frame} (VF).
A framed vertex operator algebra (FVOA) is defined in an obvious way. 

We remark that in the case $\F=\C$ the assumption $V=\oplus_{n\geq 0}V_n$ is unnecessary  \cite{DGH}.  

Let $T_r=L(\ha,0)_\F^{\otimes r}.$ Then the vertex operator subalgebra of $V$ generated by $\o_1,\ldots,\o_r$ is isomorphic to $T_r$ and is rational
by Lemma \ref{tensorrational} and Theorem \ref{harational}.  For $h_i\in\{0,\ha,\se\}$ with $i=1,...,r,$  we set 
$$L(h_1,\ldots,h_r)_\F=L(\ha,h_1)_\F\otimes\cdots \otimes L(\ha,h_r)_\F.$$
Then $L(h_1,\ldots,h_r)_\F$ is an irreducible $T_r$-module and every irreducible $T_r$-module is given in this way. Using the rationality of $T_r,$  
we decompose $V$ into a direct sum of irreducible $T_r$-modules:
$$V=\bigoplus_{h_i\in\{0,\frac{1}{2},\frac{1}{16}\}}
m_{h_1,\ldots, h_{r}}L(h_1,\ldots,h_{r})_\F$$
where the nonnegative integer
$m_{h_1,\ldots,h_r}$ is the multiplicity of $L(h_1,\ldots,h_r)_\F$ in $V$.
If $\F=\C,$  all the multiplicities are finite and $m_{h_1,\ldots,h_r}$
is at most $1$ if all $h_i$ are different from $\frac{1}{16}$ as the gradation of $V$ is given by the eigenvalues of $L(0)$ \cite{DMZ}, \cite{DGH}.  
If $\ch\F=p$ is positive, the $L(0)$ has only $p$ different eigenvalues.  Assume  $M,W$ are two irreducible $T_r$-submodules of $V$ isomorphic to $L(h_1,\ldots,h_r)_\F.$ Then there exist $s,t\geq0$ such that 
$M=\oplus_{m\geq 0}M_{h+sp+m}$ and $W=\oplus_{m\geq 0}W_{h+tp+m}$ where
$M_n=V_n\cap M$ and $h=\sum_ih_i.$  We cannot prove that $s=t=0$ from the definition. So it is not obvious that $m_{h_1,\ldots,h_r}$ is still finite.

For $d=(d_1,...,d_r)\in \Z_2^r$,    let 
$V^d$ be the sum of all irreducible submodules isomorphic to
$L(h_1,\ldots,h_r)_\F$ such that $h_i=\frac{1}{16}$ if and only if
$d_i= 1$. Let $D=D(V)=\{ d\in \Z_2^r\mid V^d\ne 0\}.$  Then
$$V=\bigoplus_{d\in D}V^d.$$

\begin{de}\label{wtc}
Fix a positive integer $r$ and a codeword $d=(d_1,...,d_r)\in \{0,1\}^r.$  Define the  support of $d$  as $supp(d)=\{i\mid d_i=1\}$. The integer $r$ is called the \emph{length} of the codeword $d$ and \emph{the weight} $|d|$ of $d$ is defined to be the cardinality of $supp(d).$ 
\end{de}

As in \cite{DGH}, we have the following lemma.
\begin{lem}\label{coded}  Let $V$ be a FVOSA. Then

(1)   $D$ is a triply even linear binary code, i.e., $|\alpha|\equiv 0\ \mod\,8$ for any $\alpha\in D$ and $D$ is a $\Z_2$-linear subspace of $\Z_2^r$;

(2) for any $d\in D,$ $V^0$ is a simple vertex operator superalgebra and each $V^d$ is an irreducible $V^0$-module. 
\end{lem}
 
The proof of Lemma \ref{coded} is similar to the same result in \cite{DGH} by using Lemma \ref{dl} and Theorem \ref{harational}.

For each  $c=(c_1,...,c_r)\in \Z_2^r,$ let $V(c)$ be the sum of the irreducible submodules
isomorphic to $L(\frac{1}{2}c_1,\ldots,\frac{1}{2}c_{r})_\F$.
Then $V^0=\bigoplus_{c\in \Z_2^r}V(c)$. Let $C=C(V)=\{c\in \Z_2^r\mid V(c)\ne 0\}$. 
The  following result  is an immediate consequence of  Theorem \ref{harational} and Lemma \ref{tensorfusion}.
\begin{lem}\label{codeC}  The set $ C$ is a linear binary code. Moreover,  $C$ is even if and only if $V^0$ is a vertex operator algebra.
\end{lem}

Next  result tells us that each $V(c)$ is an irreducible $T_r$-module.
\begin{lem}\label{l2.2}  Let $V$ be a FVOA over $\F$ with $\ch\F=p>0$, 
$p\neq 2,7$.  For any $c\in C,$ $V(c)=L(\frac{1}{2}c_1,\ldots,\frac{1}{2}c_{r})_\F.$
In particular,  the degree zero subspace $V_0$ of  $V$ is one dimensional: $V_0=\F\1$.
\end{lem}

\begin{proof}  Note that $V^0=\oplus_{c\in C}V(c).$ Using a  proof similar to that of  Lemma \ref{coded},  we know
 that $V(0)$ is a simple vertex operator algebra and each $V(c)$ is an irreducible $V(0)$-module.  Let $U$ consist of vectors $v\in V$ such that $L^i(n)v=0$ for $i=1,...,r$ and $n\geq -1.$ Then $U$ is the multiplicity of $T_r$ in $V$ and $U=\oplus_{m\geq 0}U_m$ is graded where $U_n=U\cap V_n.$  Moreover,  $V(0)=T_r\otimes U.$  Clearly, $V_0=U_0.$ The key point is to prove that $U=\F\1=V_0.$ 

We claim that $U$ is a vertex algebra. This is clear by noting that $L^i(n)u_sv=0$ for $u,v\in U,$  $s,n\in \Z$ with $n\geq	-1.$
 Moreover, $U$ is a simple vertex algebra as $V(0)$ is.  Also $U=\oplus_{i\in \Z}U_{ip}$ as $L(0)$ acts trivially on $U.$ 

 Fix $0\ne u\in U_{mp}, 0\ne v\in U_{np},$ and denote the irreducible $T_r$-modules generated by $u,v$ by $M$ and $N,$ respectively. Then $\<a_sb|a\in M, b\in N, s\in\Z\>$ is also an irreducible $T_r$-module $W.$ It is obvious that $M,N,W$ are isomorphic to $T_r$ as $T_r$-modules and  $Y(a,z)b$ for $a\in M$ and $b\in N$ is an intertwining operator of type $\left(\begin{array}{c}
T_r\\
T_r \   T_r
\end{array}\right).$    By Lemma \ref{tensorfusion} and Theorem \ref{harational},  we know that $N_{T_r, T_r}^{T_r}=1.$ 

From the definition, we know that $\1_{n}=\delta_{n,-1}Id_V.$ Thus $Y(u,z)|_N=u_{-1+sp}z^{-sp}$ for some $s\in\Z.$ In particular, $Y(u,z)v=u_{-1+sp}vz^{-sp}$ is nonzero. Take $w\in U_{ip}.$ Then from the associativity \cite{DR1}, there exists positive number $q\in\Z$ such that
$$(z_0+z_2)^qY(u,z_0+z_2)Y(v,z_2)w=(z_0+z_2)^qY(Y(u,z_0)v,z_2)w.$$
From the discussion above, $Y(u,z_0+z_2)Y(v,z_2)w=u_{-1+jp}v_{-1+lp}w(z_0+z_2)^{-jp}z_2^{-lp}$ and $Y(Y(u,z_0)v,z_2)w=(u_{-1+sp}v)_{-1+tp}wz_0^{-sp}z_2^{-tp}$ for some integers $j,l,s,t.$  As a result, we have 
$$u_{-1+jp}v_{-1+lp}w(z_0+z_2)^{-jp}z_2^{-lp}=(u_{-1+sp}v)_{-1+tp}wz_0^{-sp}z_2^{-tp}.$$
This forces $j=s=0$ and $t=l.$ Since $v\in U$ is arbitrary,  we see that $Y(u,z)=u_{-1}$ on $U$ for any
$u\in U.$ This implies that 
$$u_{-1}v_{-1}w=(u_{-1}v)_{-1}w.$$
That is, $U$ is an associative algebra over $\F$ with product $u\cdot v=u_{-1}v.$ 

Using the commutativity 
$$(z_1-z_2)^qY(u,z_1)Y(v,z_2)w=(z_1-z_2)^qY(v,z_2)Y(u,z_1)w$$
for some $q\in \Z,$ we see that $u_{-1}v_{-1}w=v_{-1}u_{-1}w.$ In particular, if $w=\1$ we have $u_{-1}v=v_{-1}u$ and $U$ is a simple,  commutative associative algebra. For $u\in U_{pm}, v\in U_{np},$ $u_{-1}v\in U_{p(m+n)}.$ 
So $U=\oplus_{n\geq 0}U_{np}$ is a  $\Z$-graded algebra and $I=\oplus_{n>0}U_{np}$ is an ideal of $U.$ From the simplicity of $U$ we conclude that $I=0$ and $U=U_0$ is a finite dimensional simple commutative associative algebra over $\F.$ This implies that $U$ is a finite field extension of $\F.$ Since $\F$ is algebraically closed,
$U=\F=\F\1.$ 

Finally, each $V(c)$ is an irreducible $V(0)$-module. So $V(c)=L(\ha c_1,...,\ha c_r)_\F.$
 \end{proof}

We remark that Lemma \ref{l2.2} was obtained in \cite{DMZ} and \cite{DGH} in the case $\F=\C.$ But the proof here is much more complicated as we cannot use the unitarity of the modules for the Virasoro algebra with central  charge $\ha$ in the current situation.

Next we deal with the multiplicities $m_{h_1,\ldots,h_r}$ in general. The same result was given in Proposition 2.5 \cite{DGH} when $\F=\C.$ The proof in \cite{DGH} works here.
\begin{lem}\label{l2.3} Let $V$ be a FVOA. Let $d\in D$ and suppose
that $(h_1,\ldots,h_r)$ and $(h_1',\ldots,h'_r)$ are $r$-tuples
with $h_i$, $h_i'\in\{0,\ha,\se\}$ such that $h_i={1 \over 16}$
(resp.~$h_i'={1 \over 16}$) if and only if $d_i=1$.
If both $m_{h_1,\ldots,h_r}$ and $m_{h_1',\ldots,h'_r}$  are nonzero
then $m_{h_1,\ldots,h_r}=m_{h_1',\ldots,h'_r}$.
That is, all
irreducible modules inside $V^d$ for $T_r$ have the same multiplicities.
\end{lem}

Recall the codes $\cal C$ and $\cal D.$ The main  result in this section is the following. 
\begin{thm}\label{t4.5} Any FVOA $V$ is rational. Moreover, 
$${C}\subset { D}^{\perp}=\{x=(x_1,...,x_r)\in \Z_2^r\mid x\cdot d=0 \text{   
for all } d\in D\}.$$
\end{thm}

The proof of Theorem 2.12 of \cite{DGH} in the case $\F=\C$ is valid here.

\section{Clifford algebras and  vertex operator algebras}
\setcounter{equation}{0}

In this section, we study the vertex operator superalgebra $V(H_\F)$ and its twisted modules for general vector space $H_\F$  of dimension $r.$ Let  $\D$ be an integral domain as before.

Let $H_{\D}=\sum_{i=1}^{r}\D a_{i}$ be a free $\D$-module of rank $r$ with a nondegenerate symmetric bilinear form $(,)$ such that  $\{a_i|
i=1,2,...r\}$ is an orthonormal basis of $H_\D.$
 The Clifford algebra $A(H_\D,d)$  associated to a codeword $d=(d_1,...,d_r)\in \{0,1\}^r$    
 is an 
associative algebra over $\D$ generated by $\{a_i(n_i)\mid 1\leq i\leq r,n_i\in {\Bbb
Z}+\frac{1}{2}(d_i+1)\}$ subject to the relation
$$[a(n),b(m)]_{+}=(a,b)\delta_{{m+n},0}$$
for $a,b\in H_\D.$
Let $A^{+}(H_\D,d)$ be the subalgebra generated by $\{a_i(n_i)|1\leq i\leq r, n_i\in
{\Bbb Z}+\frac{1}{2}(1+d_i),n_i>0 \},$ and make $\D$ a $1$-dimensional
$A^{+}(H_\D,d)$-module so that $a_i(n)1=0$
for $n>0$. We have the induced module
\begin{eqnarray*}
& &V(H_\D,d)=A(H_\D,d)\otimes_{A^{+}(H_\D,d)}\D\\
& & \ \ \ \ \ \ \ \ \ \  \ \ \,\cong \wedge_{\D} [a_i(-n_i)\mid
n_i>0 ,i=1,2,...r] \prod_{d_j=1}(\D+\D a_j(0)) ({{\rm linearly}})
 \end{eqnarray*}
so that the action of $a_i(n)$ is given by $\frac{\partial}{\partial a_i(-n)}$
if $n$ is positive and by multiplication by $a_i(n)$ if $n$ is nonpositive.
The $V(H_\D,d)$ is naturally graded by $\frac{1}{2}\Z$ with
$$V(H_\D,d)_{n+\frac{|d|}{16}}= \<a_{i_1}(-n_1)a_{i_2}(-n_2)\cdots a_{i_k}(-n_k)\mid
n_1+n_2+\cdots n_k=n \>.$$
Note that $V(H_\D,d)=V(H_\D,d)_{\bar 0}\oplus V(H_\D,d)_{\bar 1}$ where $V(H_\D,d)_{\bar i}$ is the span of
monomials whose length is congruent to $i$ modulo $2.$

For short,  we let  $V(H_\D)=V(H_\D,(0,...,0)), $ $\1=1$ and $\omega=\frac{1}{2}\sum_{i=1}^ra_i(-\frac{1}{2})a_i(-\frac{3}{2}).$ Then $V(H_\D)=V(\D a_1)\wedge\cdots \wedge V(\D a_r).$    From \cite{FFR, KW, L1, L2},  we have
\begin{thm}\label{super} (1) $(V(H_\D), Y, \1, \omega)$ is a vertex  operator superalgebra over $\D$ generated by $a(-1/2)$ for
$a\in H_\D$ and with $Y(a(-1/2),z)=a(z)=\sum_{n\in\frac{1}{2}+\Z}a(n)z^{-n-1/2}.$ Moreover, if $\D$ is a field, $V(H_\D)$ is a simple vertex operator superalgebra.

(2)  $V(H_\D)_{\bar 0}$ is a vertex operator algebra. If $\D$ is a field, both  $V(H_\D)_{\bar 0}$ and  $V(H_\D)_{\bar 1}$ are irreducible  $V(H_\D)_{\bar 0}$-modules. In particular,  $V(H_\D)_{\bar 0}$ is simple. 
\end{thm} 

To see how the vertex operator $Y(v,z)$ is defined for $v=b_{1}(-n_1-\frac{1}{2})\cdots b_{k}(-n_k-\frac{1}{2})\in V(H_\D)$,  we need
a normal ordering:
\begin{equation}\label{normalorder}
:b_1(n_1)\cdots b_k(n_k):=(-1)^{|\sigma|} b_{i_1}(n_{i_1})\cdots b_{i_k}(n_{i_k})
\end{equation}
such that $n_{i_1}\leq \cdots \leq n_{i_k}$ where  
$\sigma$ is the permutation of $\{1,...,k\}$ by sending $j$ to $i_j.$ It is easy to see that
$$Y(v,z)=:(\partial_{n_1}b_1(z))\cdots (\partial_{n_k}b_k(z)):$$
where $\partial_n=\frac{1}{n!}(\frac{d}{dz})^n.$ Note that for any $m\in \Z$ the constant
${m\choose n}$ for $n\geq 0$ is an integer. So the component operators of $Y(v,z)$ are  well defined
linear operators on $V(H_\D).$

We now define a $\D$-valued bilinear form $(\cdot,\cdot)$ on $V(H_\D)$ such that   the monomials form an orthonormal   basis.  It is clear that the form is symmetric and nondegenerate. 

\begin{prop}\label{bilinear}   Let  $a\in H_\D,$ $u,v\in V(H_\D),$ $w\in V(H_\D)_{\bar i},$ $m\in\frac{1}{2}+\Z$ and $n\in \Z.$ Then

(1)    $(a(m)u,v)=(u,a(-m)v),$

(2) $(L(n)u,v)=(u,L(-n)v),$

(3) $\frac{L(1)^n}{n!}$ is well defined if $n\geq 0,$

(4) The form is invariant:
$$(Y(w,z)u,v)=(u, Y(e^{zL(1)}(-1)^{\deg w+2(\deg w)^2+i}z^{-2\deg w}w,z^{-1})v).$$
\end{prop}

\begin{proof}  (1)  It is good enough to take $a=a_i$ for some $i.$ The result follows immediately from the definition of form.

(2) follows from (1) by noting that
$$L(n)=\frac{1}{2}\sum_{i=1}^r\sum_{j\in \frac{1}{2}+\Z}j:a_i(-j)a_i(j+n):.$$

(3) Since $\frac{L(-1)^n}{n!}$ is well defined for $n\geq 0,$ using the invariant property we see that 
$$( \frac{L(1)^n}{n!}u,v)=(u,  \frac{L(-1)^n}{n!}v)$$
and  $\frac{L(1)^n}{n!}$ is well defined.

(4) From the definition of the bilinear form,  we know the invariant property holds for $w=a(-\frac{1}{2}).$ Note that $V(H_\D)$ is generated by $a_i(-\frac{1}{2})$ for $i=1,...,r.$ It follows the proofs of Proposition 2.11 of  \cite{DLin} and Proposition 2.5 of \cite{AL} that the invariant property holds for any $w.$
\end{proof}

 We now assume that $\D=\F$ is an algebraically closed field  with $\ch \F\ne 2, 7.$ 
\begin{prop}  (1)  Set $\omega_i=\frac{1}{2}a_i(-\frac{3}{2})a_i(-\frac{1}{2})$ for $i=1,...,d.$ Then each $\omega_i$
generates a vertex operator subalgebra $\<\omega_i\>$ isomorphic to $L(\ha,0)_\F$ and $\{\omega_1,...,\o_r\}$
forms a Virasoro frame. In fact, $V(H_\F)$ is a framed vertex operator superalgebra. 

(2)  $(V(H_\F), Y, \1,\omega)$ is a holomorphic vertex operator superalgebra in the sense that $V(H_\F)$ is rational and $V(H_\F)$ is the only irreducible module for itself.

(3) We have decompositions
\begin{eqnarray*}
& &V(H_\F)\cong \bigoplus_{h_i\in\{0,\frac{1}{2}\}}L(h_1,\ldots,h_{r})_\F\label{3.1}\\
& &V(H_\F)_{\bar 0}\cong \bigoplus_{h_i\in\{0,\frac{1}{2}\},\sum_{i}h_i\in\Z}L(h_1,\ldots,h_{r})_\F\label{3.2}\\
& &V(H_\F)_{\bar 1}\cong \bigoplus_{h_i\in\{0,\frac{1}{2}\},\sum_{i}h_i\in\Z+\ha}L(h_1,\ldots,h_{r})_\F\label{3.3}.
\end{eqnarray*}
as modules for $T_r=\<\omega_i|i=1,...,r\>.$
\end{prop}

\begin{proof}  (1) follows from Theorem \ref{harational} and (3) is clear.  The rationality of $V(H_\F)$ follows immediately from Theorem \ref{t4.5}. To prove (2), let $W=\oplus_{n\in\frac{1}{2}\Z_+}W(n)$ be an irreducible $V(H_\F)$-module with $W(0)\ne 0.$ Let $Y_W(a(-\frac{1}{2}),z)=\sum_{m\in \Z}a(-\frac{1}{2})_mz^{-m-1}$ for $a\in H_\F.$ Then $W$ is an irreducible
$A(H_\F,0)$-module such that $a(-\frac{1}{2}+s)$ acts as $a(-\frac{1}{2})_{-1+s}$ for $s\in \Z.$ Clearly, $a(-\frac{1}{2}+s)W(0)=0$
for $a\in H_\F$ and $s>0.$ This implies $W$ is isomorphic to $V(H_\F)$ as $A(H_\F,0)$-modules. Since $V(H_\F)$ is generated 
by $a(-\frac{1}{2})$ for $a\in H_\F,$ $W$ is isomorphic to $V(H_\F)$ as $V(H_\F)$-modules. 
\end{proof}

Write $Y(\omega_i,z)=\sum_{n\in\Z}L_i(n)z^{-n-2}$ for $i=1,...,r.$ Set $\sigma_i=(-1)^{2L_i(0)}.$ Then
$\sigma_i$ is an automorphism of $V(H_\F)$  from the fusion rules of $L(\ha,0)_\F$-modules \cite{DR1} and Lemma \ref{tensorfusion}, and is the Miyamoto involution associated to $\omega_i$ \cite{M1, M2}.  Recall the
codeword $d$ and $V(H_\F,d).$ Let $\sigma(d)=\prod_{i, d_i=1}\sigma_i.$ Then $ \sigma(d)$ is an automorphism of $V(H_\F).$

\begin{prop} Let $d$ be as before. Then

(1) $V(H_\F)$ has an unique $\sigma(d)$-twisted module $V(H_\F)(\sigma(d))$ if $|d|$ is even and
has two inequivalent $\sigma(d)$-twisted modules $ V(H_\F)(\sigma(d))^i$ for $i=1,2$ if $|d|$ is odd.

(2) $V(H_\F,d)$ is a $ \sigma(d)$-twisted $V(H_\F)$-module such that
$$Y(a_i(-1/2),z)=\sum_{n\in\frac{1}{2}(d_i+1)+\Z}a_i(n)z^{-n-1/2}$$
 for $i=1,...,r.$

(3) We have decompositions
$$ 
V(H_\F,d)\cong \left\{\begin{array}{ll} 2^{|d|/2}V(H_\F)(\sigma(d)) & |d|\in 2\Z\\
2^{(|d|-1)/2}V(H_\F)(\sigma(d))^1\oplus 2^{(|d|-1)/2}V(H_\F)(\sigma(d))^2 & |d|\in2\Z+1.
\end{array}\right.$$
\end{prop}

\pf (1) and (2) follow from Propositions 4.3 of \cite{L2} (also see \cite{DZ}). To prove (3) consider the subalgebra $cl(d)$ of $A(H_\F,d)$ generated by $a_i(0)$ where $i\in supp(d).$ Then $cl(d)$ is a Clifford algebra
of dimension $2^{|d|}$ and is a semisimple module for itself. It is well know that $cl(d)$ is a simple algebra
with the unique irreducible module of dimension $2^{|d|/2}$ if $|d|$ is even, and is a sum of two simple algebras with two inequivalent simple modules of dimension $2^{(|d|-1)/2}.$ It is clear that any irreducible
$cl(d)$-submodule $W$ of $cl(d)$ generates an irreducible $\sigma(d)$-twisted module $A(H_\F,d)W.$ The proof
is complete. \qed

We also have the following decomposition of $V(H_{\BF},d)$ as a module for $T_r:$
$$ V(H_\F,d)=\bigoplus_{h_i\in \{0,\ha,\se\}, h_i=\se \Leftrightarrow d_i=1}2^{|d|}L(h_1,...,h_r)_\F.$$

\section{Code vertex operator algebras}\label{sec:6}

\setcounter{equation}{0}

We study the code vertex operator algebra $M_C$ for any binary even code $C$ in this section following \cite{M1, M2, LY2}. 

 Fix an  even code $C\subset \Z_2^r.$ Let  $c=(c_1,...,c_r)\in C.$ For short, we set
$\L(c)=L(h_1,...,h_r)_\F$ where $h_i=0$ if $c_i=0$ and $h_i=\ha$ if $c_i=1.$ Regard each $\L(c)$ as a subspace of
$V(H_\F)_{\bar 0}$.  Then 
$$M_C=\oplus_{c\in C}\L(c)$$
is a vertex operator algebra by Theorem \ref{harational} and Lemma \ref{tensorfusion} with $C(M_C)=C$ \cite{M2}. Moreover, any code vertex operator algebra $V$ with $ C(V)=C$ is isomorphic to $M_C$ by the uniqueness of the simple current extension \cite{DM}.

Clearly, $ \sigma(d)$ is also an automorphism of $M_C$ and $\sigma(d)|_{M_C}=1$ if and only if
$d\in C^{\perp}$, where $C^\perp= \{ \beta\in \Z_2^r\mid \alpha\cdot \beta= 0 \text{ for } \alpha\in C\}$ is the dual code of $C$ in $\Z_2^r$. 
As a result, $V(H_\F,d)$ is a $ \sigma(d)$-twisted $M_C$-module and $V(H_\F,d)$ is a $M_C$-module if and only if
$d\in C^{\perp}.$ We next decompose $V(H_\F,d)$ into a direct sum of irreducible $M_C$-modules
for $d\in C^{\perp}.$ For this purpose we consider the decomposition
$$V(\F a_i, d_i)=V(\F  a_i, d_i)_{\bar 0}\oplus V(\F a_i, d_i)_{\bar 1}$$
where
$$ V(\F a_i,d_i)_{\bar s}=\<a_{i}(-n_1)a_{i}(-n_2)\cdots a_{i}(-n_k)\mid
k\equiv s (\mod 2), n_i\geq 0, n_i\in \ha(1+d_i)+\Z \>$$
for $s=0,1.$ Then
$V(\F a_i,d_i)_{\bar s}\cong L(\ha,\se)$ if $d_i=1$ for $s=0,1$ (cf. \cite{KR}, \cite{DR2}). Then
$$V(H_\F,d)=\bigoplus_{s=(s_1,...,s_r)\in \Z_2^r}V(\F a_1,d_1)_{s_1}\wedge\cdots \wedge V(\F a_r,d_r)_{s_r}.$$
For each coset $(x_1,...,x_r)+C\in \Z_2^r/C$, we set
$$V(H_\F,d,C+(x_1,...,x_r))=\bigoplus_{s=(s_1,...,s_r)\in(x_1,...,x_r)+C }V(\F a_1,d_1)_{s_1}\wedge\cdots \wedge V(\F a_r,d_r)_{s_r}.$$
It follows immediately that
\begin{equation*}\label{decom}
V(H_\F,d)=\bigoplus_{(x_1,...,x_r)+C\in \Z_2^r/C}V(H_\F,d,C+(x_1,...,x_r)).
\end{equation*}

Let  $cl(d)$ be a subalgebra of  $A(H_\F,d)$ generated by $a_i(0)$ with $d_i=1.$  Then  $cl(d)$ is a finite dimensional semisimple associative algebra which has one simple module if $d$ is even and two inequivalent simple modules if $d$ is odd. 
Let $K(d)=\{c\in C\mid
supp(c)=\{i_1,...,i_k\}\subset supp(d)\}$ and let $E(d)$ be a maximal subcode of $K(d)$ such that
$E(d)\subset E(d)^{\perp}.$  For $c\in K(d)$, we set $e^c=2^{k/2}a_{i_1}(0)\cdots a_{i_k}(0)\in cl(d).$   
Then $G(d)=\{{\pm }e^c\mid c\in K(d)\}$  is a finite group 
and the group algebra $\F[G(d)]=\oplus_{c\in K(d)}e^c$ is a semisimple 
subalgebra of $cl(d).$   Moreover, $A(d)=\{\pm e^c\mid c\in E(d)\}$ is a maximal abelian subgroup of $G(d).$ 

\begin{thm}\label{tt1} Let $d\in C^{\perp}$ and $(x_1,...,x_r)+C\in \Z_2^r/C.$ 

(1)  $V(H_\F,d,C+(x_1,...,x_r))$ is a direct sum of $|E(d)|$ irreducible $M_C$-modules and each irreducible  submodule is a direct sum of
$|C/E(d)|$ irreducible $T_r$-modules. Moreover, each irreducible $M_C$-submodule is determined by a character of $A(d).$

(2) Every irreducible $M_C$-module is obtained in this way.
\end{thm}

\begin{proof}  (1) Clearly, $V(H_\F,d,C+(x_1,...,x_r))$ is an $M_C$-module.  Let $x=(x_1,...,x_r)$ and set
$$W=\bigoplus_{t=(t_1,...,t_r)\in x+E(d)}V(\F a_1,d_1)_{t_1}\wedge\cdots \wedge V(\F a_r,d_r)_{t_r}.$$ 
Note that if $d_i=0$ then $t_i=x_i.$   Also, $V(\F a_i,0)_{t_i}$ is isomorphic to $L(\ha,\frac{x_i}{2})_\F,$ and if $d_i=1,$ 
 $V(\F a_i,1)_{t_i}$ is isomorphic to $L(\ha,\se)_\F.$  This gives 
$$W\cong |E(d)|L(h_1,...,h_r)_\F$$
as $T_r$-module where $h_i=\se$ if $i\in supp(d),$ $h_i=\frac{x_i}{2}$ if $i\notin supp(d).$
It is evident that the top level $T(W)$ of  $W$ has a basis $$e^c\prod_{d_i=1,x_i=1}a_i(0)\prod_{ d_i=0, x_i=1}a_i(-1/2)$$ for $c\in E(d).$
 Moreover,
$T(W)$ is isomorphic to $\F[A(d)]/(-e^0+1)$ as $A(d)$-modules where $(-e^0+1)$ is the ideal of  $\F[A(d)]$
generated by  $-e^0+1.$  Since $A(d)$ is an abelian group, $T(W)=\oplus_{\lambda}T(W)_{\lambda}$ is a sum of irreducible  $A(d)$-modules $ T(W)_{\lambda}$ with the character $\lambda$ such that $\lambda(-e^0)=-1.$ Note that the module for the code vertex operator algebra $M_{E(d)}$ generated by $T(W)_{\lambda}$ is an irreducible $T_r$-module isomorphic to $L(h_1,...,h_r)_\F.$  We denote this  $M_{E(d)}$-module by $M_{E(d)}(x,\lambda).$

Now consider the induced module $\Ind_{A(d)}^{G(d)}T(W)_{\lambda}$ which is an irreducible $G(d)$-module by
\cite{FLM} and whose dimension is $[K(d):E(d)].$ Then the $M_{K(d)}$-module generated by  $\Ind_{A(d)}^{G(d)}T(W)_{\lambda}$  is an irreducible $M_{K(d)}$-module as $M_{K(d)}$ is rational \cite{DR1}.  Clearly,  this
 irreducible is isomorphic to $[K(d):E(d)]L(h_1,...,h_r)_\F$ as $T_r$-module.  We denote this irreducible  $M_{K(d)}$-module 
by $M_{K(d)}(x,\lambda).$ 

Let $M_C( d, x, \lambda)$ be the $M_C$-module generated by $T(W)_{\lambda}.$ Then 
$$M_C( d, x, \lambda)=\sum_{s+K(d)\in C/K(d)} M_{K(d)+s}\cdot M_{K(d)}(x,\lambda)$$
where $ M_{K(d)+s}\cdot M_{K(d)}(x,\lambda)$ is spanned by $u_nM_{K(d)}(x,\lambda)$ for $u\in  M_{K(d)+s}$ and 
$n\in \Z.$ Clearly,  $ M_{K(d)+s}\cdot M_{K(d)}(x,\lambda)$ is also the fusion product of $ M_{K(d)+s}$ with $M_{K(d)}(x,\lambda)$
as   $ M_{K(d)+s}$ is a simple current.  Moreover, 
 $$M_{K(d)+s}\cdot M_{K(d)}(x,\lambda)\cong M_{K(d)}(s+x,\lambda)$$
is an irreducible $M_{K(d)}$-module. Note that if $s+K(d)\ne t+K(d)$ for $s,t\in C$ then each irreducible 
$T^r$-module in  $M_{K(d)}(s+x,\lambda)$ and   $M_{K(d)}(t+x,\lambda)$ has different $\frac{1}{2}$ positions. This shows that
$M_C( d, x, \lambda)$ is an irreducible $M_C$-module. 

From the definitions of $V(H_\F,d, C+x)$ and $W$,  we immediately see that $V(H_\F,d, C+x)=M_C\cdot W.$ This leads to the
decomposition 
\begin{equation}\label{decom1}V(H_\F,d,C+x)=\oplus_{\lambda} M_C( d, x, \lambda).
\end{equation}

The proof of (2) is similar to that given in \cite{M1, M2, LY2} when $\F=\C.$
\end{proof}

\section{$\D$-forms}
\setcounter{equation}{0}

From now on,  we assume $\D=\Z[\frac{1}{2}].$ We will construct some $\D$-form for framed vertex operator algebras over $\C$.  The main idea is to investigate  the $\D$-forms of code vertex operator algebras and some of their  irreducible modules.  In this section,  we first construct a $\D$-form for any simple current of a code  vertex operator algebra over $\C.$  Then we will study certain intertwining operators over $\D$ in Section 12 and  construct a $\D$-form for any framed vertex operator algebra over $\C$ in Section 13.

Fix an even binary code $C\subset  \Z_2^r.$  Set $(M_{C+x})_{\D}=M_{C+x}\cap V(H_\D)$ for any $x\in \Z_2^r.$  It is clear that $(M_C)_{\D}$ is vertex operator algebra over $\D$ and $(M_{C+x})_{\D}$ is an  $(M_C)_{\D}$-module.
Clearly, $(M_{C+x})_{\D}$ is $\D$-form of $M_{C+x}$ over$(M_C)_\D.$
Recall the invariant bilinear from $(\cdot,\cdot)$ on $V(H_\D).$  We also denote
the restriction of the form to $(M_{C+x})_{\D}$ by $(\cdot,\cdot).$ We call 
$$(M_{C+x})_{\D}^*=\{u\in M_{C+x}\mid
(u,(M_{C+x})_{\D})\subset \D\}$$
 the dual of  $(M_{C+x})_{\D}$ in $M_{C+x}$ with respect to
the form. 

\begin{lem}\label{form1} The form on  $(M_{C+x})_{\D}$  is self dual. That is,  $(M_{C+x})_{\D}^*=(M_{C+x})_{\D}.$
\end{lem}
\begin{proof} From the definition we see that
$$(M_{C+x})_\D=\oplus_{c=(c_1,...,c_r)\in C+x}V(\D a_1)_{\bar c_1}\wedge\cdots\wedge V(\D a_r)_{\bar c_r}$$
and $(M_{C+x})_\D$ has a $\D$-base consisting of monomials in $M_{C+x}.$ Since these monomials form an orthonormal base 
of $(M_{C+x})_\D,$ the result follows. 
\end{proof}

Next we study the $\D$-form of $M_C$-module with the  $\se$ position code $d\in C^{\perp}.$ Recall the irreducible 
$M_C$-module $M_C(d,x,\lambda)$ over $\C$ and the code $E(d)$  from Section 10. For our purpose,  we only deal with the case when $M_C(d,x,\lambda)$ is a self-dual simple current. The following result  comes from \cite[Corollary 4 and Proposition 5]{LY2}.

\begin{lem}\label{formly} Let $K(d)$ and $E(d)$ be as in Section 10. An $M_C$-module $M_C(d,x,\lambda)$ is  a simple current if and only if $E(d)$ is a self-dual subcode of $K(d)$. Furthermore, if $M_C(d,x,\lambda)$ is self-dual, then we can choose $E(d)$ to be doubly even.  
\end{lem}

Here we present a useful fact about an elementary abelian $2$-group, i.e., a finite group which all non-identity elements have order $2$.
\begin{lem}\label{2group} Let $G$ be an elementary abelian $2$-group.  
Let $k=o(G).$  Then $\C[G]$ has a basis $\{u_1,...,u_k\}\subset \D[G]$ such that each $\D u_i$ is a $\D[G]$-module.
\end{lem}

\begin{proof} We proceed by  induction on $k.$ If $k=2$ then $G=\{e,a\}.$ We can take  $u_1=e+a,$ and $u_2=e-a.$ Assume the result  for $k$ we now prove the result for $2k.$ Let $G$ be generated by $\{a_1,...,a_m\}$ such that $2^{m-1}=k.$ Let $G_1$ be the subgroup of $G$ generated by $\{a_1,...,a_{m-1}\}.$ Then $\C[G]=\C[G_1](\C e+\C a_m).$ Let $v_1,...,v_k$ be a basis
of $\C[G_1]$ such that $\D v_i$ is a $\D[G_1]$-module. Set $u_i=v_i(e+a_m)$ for $i=1,...,k$ and $u_i=v_i(e-a_m)$ for 
$i=k+1,...,2k.$ It is clear $\D u_i$ is a $\D [G]$-module.
\end{proof}

 We remark  that only numbers $0, 1, -1$ are used in the proof of Lemma \ref{2group}, so Lemma \ref{2group} is valid with $\D$ replaced by $\Z.$

We can now give our key Lemma.
\begin{lem}\label{klem} If $M_C(d,x,\lambda)$ is a simple current, then $M_C(d,x,\lambda)$ has a $\D$-form 
$M_C(d,x,\lambda)_\D$ such that $M_C(d,x,\lambda)_\D$ is an $(M_C)_\D$-module. 
\end{lem}
\begin{proof} From the definition,  we notice that  $V(H_\D,d)$ is a $\D$-submodule  of $V(H_\C,d).$ In fact, the monomials form
a $\D$-base of   $V(H_\D,d).$  We claim, in fact, that   $V(H_\D,d)$ is a $\D$-form of $V(H_\C,d)$ over $V(H_\D).$ It is good enough to show that for any $u\in  V(H_\D)$ and $m\in\Z,$
$u_m V(H_\D,d)\subset V(H_\D,d).$ Let $u=a_{i_1}(-n_1-1/2)\cdots a_{i_k}(-n_k-1/2)$ with $n_j\in\Z_{+}$ and we prove the claim by induction on $k.$  If $k=1$ then $Y(a_i(-n-1/2),z)=\partial_na_i(z)$ where 
$$a_i(z)=\sum_{n\in\frac{1}{2}(d_i+1)+\Z}a_i(n)z^{-n-1/2}.$$
It follows immediately that  $Y(a_i(-n-1/2),z)V(H_\D,d)\subset V(H_\D,d)[[z^{1/2}, z^{-1/2}]].$

Now assume that the claim is true for $k.$ Using the twisted Jacobi identity, we see that for any $w\in V(H_\D,d),$ 
$$(a_{i_1}(-n_1-1/2)\dots a_{i_{k+1}}(-n_{k+1}-1/2))_mw=\sum_j z_ja_{i_1}(p_j)v_{q_j}w$$
for some $z_j\in \D$ where $v=(a_{i_2}(-n_2-1/2)\cdots a_{i_{k+1}}(-n_{k+1}-1/2).$  So the claim is true for $k+1.$ One can  also use the explicit expression of $Y(u,z)$
from \cite{FFR} to  see that each component $u_m$ is a $\D$ linear combination of operators 
$a_{i_1}(m_1)\cdots a_{i_k}(m_k).$ 

Recall that $M_C(d,x,\lambda)$ is an $M_C$-submodule of $V(H_\C,d).$ Let 
$$M_C(d,x,\lambda)_\D=M_C(d,x,\lambda)\cap V(H_\D,d).$$
 Clearly, $M_C(d,x,\lambda)_\D$ is  an $(M_C)_\D$-module. It remains to show that  $M_C(d,x,\lambda)_\D$ is nonzero.  By Lemma \ref{formly},  $E(d)$ is a doubly even selfdual subcode of $K(d).$ So the weight of each codeword of $E(d)$ is a multiple of $4.$ Since  $a_i(0)a_j(0)+a_j(0)a_i(0)=\delta_{i,j},$ $A(d)$ is an abelian group such that each element has order  less than or equal to $2.$ Moreover, $G=\{e^c\mid c\in E(d)\}$ is a subgroup of $A(d)$ of index $2$ such that 
any irreducible $A(d)$-module such that $-e^0$ acts as $-1$ is an irreducible $G$-module. 

 We need to recall how  $M_C(d,x,\lambda)$ is constructed in the proof of Theorem \ref{tt1}.  In particular, the top level $T(W)$
of $W$ has a basis  $e^c\prod_{d_i=1,x_i=1}a_i(0)\prod_{ d_i=0, x_i=1}a_i(-1/2)$ for $c\in E(d).$ Clearly, $T(W)$ is isomorphic to $\C[G]$ as $G$-module
by identifying  $e^c$ with $e^c\prod_{d_i=1,x_i=1}a_i(0)\prod_{ d_i=0, x_i=1}a_i(-1/2)$  for $c\in E(d).$ From the definition of $V(H_\D,d),$ it is evident that
 $e^c\prod_{d_i=1,x_i=1}a_i(0)\prod_{d_i=0, x_i=1}a_i(-1/2)$ lies in $ V(H_\D,d)$ for $c\in E(d).$   By Lemma \ref{2group},  there is basis
$$\{u_1,....,u_k\}\subset \D[G]\prod_{i, d_i=0, x_i=1}a_i(-1/2)$$
 of $T(W)$ such that $\C u_i$ is an irreducible $G$-module. 
Let $T(W)_{\lambda}=\C u_i$ for some $i.$ Then $u_i\in M_C(d,x,\lambda)_\D.$ We conclude that $M_C(d,x,\lambda)_\D$
is a $\D$-form over $(M_C)_\D.$ 
\end{proof}

\begin{rem}\label{r0} From the proof of Lemma \ref{klem},  we see that $\C u_i$ is an irreducible $G$-module. Let $\lambda_i$ be the corresponding character. Then 
$$V(H,d,C+x)=\oplus_{i=1}^kM_C(d,x,\lambda_i)$$
and
$$V(H_\D, d,C+x)=\oplus_{i=1}^kM_C(d,x,\lambda_i)_\D.$$
\end{rem}

To prove any framed vertex operator algebra over $\C$ has a $\D$-form, we need an invariant bilinear form on $V(H_\C, d)$ and
$M_C(d,x,\lambda).$ Recall that $V(H_\D,d)$ has a monomial basis. We define a bilinear form $(\cdot,\cdot)$ on $V(H_\C, d)$ such that the monomials form an orthogonal basis. More precisely, let $u=a_{i_1}(0)\cdots a_{i_k}(0)v$ where $v$ is a monomial 
without any $a_i(0)$ and $i_p\ne i_q$ if $p\ne q.$ We define the square length of $u$ is $\frac{1}{2^k}.$ It is trivial to show that
for any $u,v\in V(H_\C,d),$ $i=1,...,r$ and $n\in\frac{1}{2}\Z,$ $(a_i(n)u,v)=(u,a_i(-n)v).$  The next result is an analogue of Proposition \ref{bilinear}.  

\begin{lem}\label{invform}  Let  $u,v\in V(H_\D,d),$ $w\in V(H_\D)_{\bar i},$ $m\in\frac{1}{2}+\Z$ and $n\in \Z.$ Then

(1) The $\D$-bilinear  $\D$-valued form  $(\cdot,\cdot)$ on $V(H_\D,d)$ is positive definite . 

(2)  $ V(H_\D,d)$ is self dual in $V(H_\C,d)$ in the sense that 
$$\{w\in V(H_\C,d)|(w, V(H_\D,d))\subset \D\}=V(H_\D,d).$$

(3) $(L(n)u,v)=(u,L(-n)v),$

(4) $\frac{L(1)^n}{n!}$ is well defined on $V(H_\D,d)$ if $n\geq 0,$

(5) The form is invariant:
$$(Y(w,z)u,v)=(u, Y(e^{zL(1)}(-1)^{\deg w+2(\deg w)^2+i}z^{-2\deg w}w,z^{-1})v).$$
\end{lem}
\begin{proof}  (1) and (2) are  clear from the definition of the form. 

(3)  We know from \cite{KR} that 
$$L(0)=\frac{|d|}{16}+ \frac{1}{2}\sum_{d_i=0}\sum_{j\in \frac{1}{2}+\Z}j:a_i(-j)a_i(j): +\frac{1}{2}\sum_{d_i=1}\sum_{j\in \Z}j:a_i(-j)a_i(j): $$
and 
$$L(n)= \frac{1}{2}\sum_{d_i=0}\sum_{j\in \frac{1}{2}+\Z}j:a_i(-j)a_i(j+n): +\frac{1}{2}\sum_{d_i=1}\sum_{j\in \Z}j:a_i(-j)a_i(j+n): $$
if $n\ne 0$ where the normal order  is defined as in (\ref{normalorder}).  Using the explicit expression of $L(n)$ gives
 $(L(n)u,v)=(u,L(-n)v).$

(4) A special case of the commutator formula 
$$[L(m), a_i(s)]=-(s+\frac{m}{2})a_i(m+s)$$
 gives  $[L(1), a_i(s)]=-(s+\frac{1}{2})a_i(1+s).$ So 
$$\frac{L(1)^n}{n!}a_i(-s)=\frac{(-1)^n(-s+1/2)\cdots (-s+n-1+1/2)}{n!}a_i(-s+n)\1$$
 for $s>0.$ 
Clearly, $\frac{(-s+1/2)\cdots (-s+n-1+1/2)}{n!}\in \D$ and $\frac{L(1)^n}{n!}a_i(-s)$ lies in $V(H_\D,d).$ Now assume that $\frac{L(1)^n}{n!}v$ is well defined for $v\in V(H_\D,d)$ and $u=a_i(-s)v.$ Then 
$$\frac{L(1)^n}{n!}u=\sum_{j=0}^n( \frac{(ad L(1))^j}{j!}a_i(-s))\frac{L(1)^{n-j}}{(n-j)!}v$$
as $ \frac{(ad L(1))^j}{j!}a_i(-s)=\frac{(-1)^j(-s+1/2)\cdots (-s+j-1+1/2)}{j!}a_i(-s+j)$ and $\frac{(-1)^j(-s+1/2)\cdots (-s+j-1+1/2)}{j!}\in \D.$ Since $a_j(-s+j)$ preserves $V(H_\D,d),$ $\frac{L(1)^n}{n!}v$ is an element of $V(H_\D,d).$

(5) The proof is similar to that given in Proposition \ref{bilinear} (4).  
\end{proof}

\begin{cor}\label{dform} Assume that    $M_C(d,x,\lambda)$ is a simple current. There is a positive definite $\D$-valued, invariant bilinear form on $(M_C(d,x,\lambda))_\D.$ Moreover, the dual of   $M_C(d,x,\lambda)_\D$ in $M_C(d,x,\lambda)$ is itself.
\end{cor}
\begin{proof}  It is clear from Lemma \ref{invform} that the restriction of the form to  $M_C(d,x,\lambda)_\D$ is positive definite and invariant.  In order to prove that  $M_C(d,x,\lambda)_\D$ is self dual in $M_C(d,x,\lambda)$,  we recall the decomposition (\ref{decom}). Since the dual of $V(H_\D,d)$ in $V(H_\C,d)$ is itself by Lemma \ref{invform} (2), and each $V(H_\C,d, C+x)$ has a basis
consisting of monomial basis, we see that $V(H_\D,d,C+x)$ is self dual in  $V(H_\C,d, C+x).$ By 
(\ref{decom1})  and the proof of Lemma \ref{klem}
$$ V(H_\D,d,C+x)=\oplus_{\lambda}(M_C(d,x,\lambda))_\D.$$
This immediately implies that $M_C(d,x,\lambda)_\D$ is self dual  in $M_C(d,x,\lambda).$
\end{proof}

\section{Intertwining operators over $\D$}

 We construct  various intertwining operators associated to the vertex operator algebra $V(\C a)$
and $V(H_\D)$ and $(M_C)_\D$ over  $\D=\Z[\frac{1}2]$ in this section. These intertwining operators will be used later to construct $\D$-form
for a framed vertex operator algebra over $\C.$  

Let $V=V(\C a)\oplus V(\C a,1)$ and define a nondegenerate symmetric bilinear form $(\cdot,\cdot)$ on $V$ such that the restrictions of the form to $V(\C a)$
and $ V(\C a,1)$  are as before, and $(V(\C a),V(\C a,1))=0.$ We also set $V_{\D}=V(\D a)\oplus V(\D a,1).$

Define a linear map 
$$Y: V\to (\End V)[[z^{1/8},z^{-1/8}]]$$
such that $Y(v,z)$ for $v\in V(\C a)$ is exactly the vertex operator which defines the vertex operator subalgebra structure on $V(\C a)$
and the $ \sigma$-twisted $V(\C a)$-module structure on $V(\C a,1)$ where $\sigma=(-1)^{2L(0)}$ is the canonical  automorphism of $V(\C a).$ For $u\in V(\C a, 1)_{\bar s}$ and  $v\in V(\C a)_{\bar r},$ we define $Y(u,z)v=(-1)^{rs}e^{L(-1)z}Y(v,-z)u$.   Now assume that $v\in V(\C a,1)_{\bar r}$. We define $Y(u,z)v$  from the identity
$$(Y(u,z)v,w)=(-1)^{st}(v, Y(e^{L(1)z}(-1)^{L(0)-1/16}z^{-2L(0)}u, z^{-1})w)$$
for any $w\in V(\C a)_{\bar t}.$ Using the notation from Section 9, we also denote $V(\C a)$ by $V(\C a,0).$ 
\begin{lem}\label{l9.0} Let $d_1,d_2,s,t=0,1$ and $u\in V(\C a,d_1)_{\bar s},$ $v\in V(\C a, d_2)_{\bar t}.$ Then
$Y(u,z)v\in V(\C a, d_1+d_2)_{\overline{s+t}}[[z^{1/8},z^{-1/8}]].$ 
\end{lem}
\begin{proof} If $d_1=0,$ the conclusion is clear from the construction of $ \sigma^{d_2}$-$V(\C a)$-twisted module $V(\C a,d_2).$
For the case $d_1=1$,  we first notice that  $(V(\C a, d)_{\bar s_1}, V(\C a, d)_{\bar t_1})=0$ for any $d=0,1$ 
if $s_1\ne t_1.$ The result now follows from the  definition of $Y(u,z)v$ and the result with $d_1=0$.  \end{proof}

\begin{lem}\label{l9.1} The restriction of $Y$ to $V_{\D}$ defines a linear map
 $$V_\D \to  (\End_\D V_{\D})[[z^{1/8},z^{-1/8}]].$$
\end{lem}
\begin{proof} From the proof of Lemma \ref{klem} we know that  $Y(u,z)v\in V_\D[[z^{1/8}, z^{-1/8}]]$ for $u\in V(\D a)$ and $v\in V_{\D}.$
Since $\frac{1}{n!}L(-1)^n$ preserves $V_{\D}$ for $n\geq 0$,  we see  immediately from the definition that  $Y(u,z)v\in V_\D[[z^{1/8}, z^{-1/8}]]$ 
for  $u\in V(\D a, 1)$ and $v\in V(\D a).$
Finally we assume that $u\in V(\D a,1)_{\bar r}$ and $v\in V(\D a,1)_{\bar s}.$ Let $w\in V(\D a)_{\bar t}.$ Then
$$(Y(u,z)v,w)=(-1)^{st}(v, Y(e^{L(1)z}(-1)^{L(0)-1/16}z^{-2L(0)}u, z^{-1})w)\in \D[[z^{1/8}, z^{-1/8}]].$$
Since $V(\D a)$ is selfdual in $V(\C a)$, we conclude that   $Y(u,z)v\in V_\D[[z^{1/8}, z^{-1/8}]]$ as desired.
\end{proof}

\begin{lem}\label{la1} (1)  $Y(v,z)w$ for $v\in V(\C a, 1)$ and $w\in V(\C a)$ is an intertwining operator of type  $\left(\begin{array}{c}
V(\C a,1)\\
V(\C a,1)\ V(\C a)
\end{array}\right).$

(2)  $Y(u,z)v$ for $u\in V(\C a, 1)$ and $v\in V(\C a, 1)$ is an intertwining operator of type  $\left(\begin{array}{c}
V(\C a)\\
V(\C a,1)\ V(\C a, 1)
\end{array}\right).$
\end{lem}
\begin{proof} (1)  First the same arguments from \cite{FHL} and \cite{X}  show that $Y(L(-1)v,z)w=\frac{d}{dz}Y(v,z)w$ for
$v\in V(\C a,1)$ and $w\in V(\C a).$  It remains  to prove the following Jacobi identity for $u\in V(\C a)_{\bar r},$ $v\in V(\C a,1)_{\bar s}$ and $w\in V(\C a)_{\bar t}:$
\[
z_{0}^{-1}\left(\frac{z_1-z_2}{z_0}\right)^{r/2}\delta\left(\frac{z_{1}-z_{2}}{z_{0}}\right)Y(u,\ z_{1})Y(v,\ z_{2})w
\]
\[-(-1)^{rs}z_{0}^{-1}\left(\!\frac{-z_2\!+\!z_1}{z_0}\!\right)^{r/2}\delta\left(\frac{-z_{2}+z_{1}}{z_{0}}\right)Y(v,\ z_{2})Y(u,\ z_{1})w
\]
\[
=z_{2}^{-1}\delta\left(\frac{z_{1}-z_{0}}{z_{2}}\right)Y(Y(u,\ z_{0})v,\ z_{2})w.
\] 
From the definition, it is equivalent to show that 
\[
(-1)^{st}z_{0}^{-1}\left(\frac{z_1-z_2}{z_0}\right)^{r/2}\delta\left(\frac{z_{1}-z_{2}}{z_{0}}\right)e^{z_2L(-1)}Y(u,z_{1}-z_2)Y(w, - z_{2})v
\]
\[-(-1)^{rs+s(r+t)}z_{0}^{-1}\left(\!\frac{-z_2\!+\!z_1}{z_0}\!\right)^{r/2}\delta\left(\frac{-z_{2}+z_{1}}{z_{0}}\right)e^{z_2L(-1)}Y(Y(u, z_1)w, -z_2)v
\]
\[
=(-1)^{(r+s)t}z_{2}^{-1}\delta\left(\frac{z_{1}-z_{0}}{z_{2}}\right)e^{z_2L(-1)}Y(w,-z_2)Y(u,\ z_{0})v,
\] 
or 
\[
z_{0}^{-1}\left(\frac{z_1-z_2}{z_0}\right)^{r/2}\delta\left(\frac{z_{1}-z_{2}}{z_{0}}\right)Y(u,z_{1}-z_2)Y(w, - z_{2})v
\]
\[
-(-1)^{rt}z_{2}^{-1}\delta\left(\frac{z_{1}-z_{0}}{z_{2}}\right)Y(w,-z_2)Y(u,\ z_{0})v,
\] 
\[=z_{0}^{-1}\left(\!\frac{-z_2\!+\!z_1}{z_0}\!\right)^{r/2}\delta\left(\frac{-z_{2}+z_{1}}{z_{0}}\right)Y(Y(u, z_1)w, -z_2)v.
\]
Note that $z^{r/2}Y(u,z)$ only involves integral powers of $z.$ So 
\[
z_{0}^{-1}\left(\frac{z_1-z_2}{z_0}\right)^{r/2}\delta\left(\frac{z_{1}-z_{2}}{z_{0}}\right)Y(u,z_{1}-z_2)=z_1^{-1}\delta\left(\frac{z_{0}+z_{2}}{z_1}\right)Y(u,z_0).
\]
 Also
\[z_{0}^{-1}\left(\!\frac{-z_2\!+\!z_1}{z_0}\!\right)^{r/2}\delta\left(\frac{-z_{2}+z_{1}}{z_{0}}\right)\\
=(-z_2)^{-1}\left(\!\frac{z_0\!-\!z_1}{-z_2}\!\right)^{-r/2}\delta\left(\frac{z_0-z_{1}}{-z_2}\right)
\]
by formula (7.1) of \cite{DL}.
Thus we need to show that 
\[
z_1^{-1}\delta\left(\frac{z_{0}+z_{2}}{z_1}\right)Y(u,z_0)Y(w, - z_{2})v
\]
\[
-(-1)^{rt}z_{1}^{-1}\delta\left(\frac{z_{2}+z_{0}}{z_{1}}\right)Y(w,-z_2)Y(u,\ z_{0})v,
\] 
\[=(-z_2)^{-1}\left(\!\frac{z_0\!-\!z_1}{-z_2}\!\right)^{-r/2}\delta\left(\frac{z_0-z_{1}}{-z_2}\right)Y(Y(u, z_1)w, -z_2)v
\]
which is the Jacobi identity in the definition of $\sigma$-twisted module.

(2) Now let $u\in V(\C a)_{\bar r}$, $v\in V(\C a ,1)_{\bar s}$ and $w\in V(\C a,1).$ It follows from the proof of Theorem 4.4  of \cite{X}
that the following Jacobi identity holds:
 \[
z_{0}^{-1}\left(\frac{z_1-z_2}{z_0}\right)^{r/2}\delta\left(\frac{z_{1}-z_{2}}{z_{0}}\right)Y(u,\ z_{1})Y(v,\ z_{2})w
\]
\[-(-1)^{rs}z_{0}^{-1}\left(\!\frac{z_2\!-\!z_1}{-z_0}\!\right)^{r/2}\delta\left(\frac{z_{2}-z_{1}}{-z_{0}}\right)Y(v,\ z_{2})Y(u,\ z_{1})w
\]
\[
=z_{2}^{-1}\left(\frac{z_{1}-z_{0}}{z_{2}}\right)^{-r/2}\delta\left(\frac{z_{1}-z_{0}}{z_{2}}\right)Y(Y(u,\ z_{0})v,\ z_{2})w.
\]
Also see \cite{X} for the $L(-1)$-derivation property.
\end{proof}

Recall from Section 9 that $H=\oplus_{i=1}^r\C a_i$ is vector space with  a nondegenerate symmetric bilinear form $(,)$ such that  $\{a_i\mid i=1,...,r\}$ form an orthonormal basis. Let $d=(d_1,...,d_r), e=(e_1,...,e_r)\in \Z_2^r.$ Let $\sigma(d)$ and $\sigma(e)$ be the associated automorphisms of the vertex operator superalgebra $V(H).$ Then  $V(H,d)$ is a $ \sigma(d)$-twisted $V(H)$-module.

We now define a linear map:
$$Y: V(H,d)\to (\Hom(V(H,e),V(H,d+e))[[z^{1/8},z^{-1/8}]]$$
such that 
$Y(u^1\wedge \cdots \wedge u^r,z)=Y(u^1,z)\cdots Y(u^r,z)$
where $u^i\in V(\C a_i,d_i)$ for $i=1,...,r$ and $Y(u^i,z)$ is defined as before. The following 
result is an immediate consequence of Lemma \ref{l9.1}.
\begin{lem}\label{l9.3} The restriction of $Y$ to $V(H_{\D},d)$ defines a linear map
 $$V(H_{\D},d) \to (\Hom(V(H_\D,e),V(H_\D,d+e))[[z^{1/8},z^{-1/8}]].$$
\end{lem}

Here is an extension of Lemma \ref{la1}:
\begin{lem}\label{l9.4} The linear map 
$$Y: V(H,d)\to (\Hom(V(H,e),V(H,d+e))[[z^{1/8},z^{-1/8}]]$$
defines  an intertwining operator of type  $\left(\begin{array}{c}
V(H,d+e)\\
V(H,d)\ V(H,e)
\end{array}\right).$
\end{lem}
\begin{proof} We only prove the following Jacobi identity 
 for any $u\in V(H)^{(j_1,j_2)}_{\bar s},\ v\in V(H,d)_{\bar t}$
\[
	z_{0}^{-1}\left(\frac{z_1-z_2}{z_0}\right)^{j_1/2}\delta\left(\frac{z_{1}-z_{2}}{z_{0}}\right)Y(u,\ z_{1})Y(v,\ z_{2})
\]
\[
-(-1)^{s+t}z_{0}^{-1}\left(\!\frac{-z_2\!+\!z_1}{z_0}\!\right)^{j_1/2}\delta\left(\frac{-z_{2}+z_{1}}{z_{0}}\right)Y(v,\ z_{2})Y(u,\ z_{1})
\]
\[=z_{2}^{-1}\left(\frac{z_1-z_0}{z_2}\right)^{-j_2/2}\delta\left(\frac{z_{1}-z_{0}}{z_{2}}\right)Y(Y(u,\ z_{0})v,\ z_{2}).
\]
where $V(H)^{(j_1,j_2)}$ is the common eigenspace for $\sigma(d), \sigma(e)$ with eigenvalues 
$(-1)^{j_1}, (-1)^{j_2}$ respectively.  

From \cite{DL}, we know that the Jacobi identity is equivalent to
the commutativity
$$(z_1-z_2)^{j_1/2+n}Y(u,z_1)Y(v,z_2)=(-z_2+z_1)^{j_1/2+n}Y(v,z_1)Y(u,z_2)
$$
for some positive integer $n,$ and the associativity
$$	(z_0+z_2)^{j_2/2+m}Y(u,z_0+z_2)Y(v,z_2)w=(z_2+z_0)^{j_2/2+m}Y(Y(u,z_0)v,z_2)w
$$
where $w\in V(H,e)$ and  nonnegative integer $m$ which depends on $u$ and $w$ only.
Note  that 
$$Y(u^1\wedge\cdots \wedge u^r,z)=Y(u^1,z)\cdots Y(u^r,z)$$
for $u^i\in V(\C a_i,x_i)$ and $x_i=0,1$ and that $Y(u^i,z)Y(u^j,z)=(-1)^{st}Y(u^j,z)Y(u^i,z)$
if $i\ne j$ and $u^i\in V(\C a_i,x_i)_{\bar s}$ and  $u^j\in V(\C a_j,x_j)_{\bar t}.$ 
Using the Jacobi identity obtained in the proof of Lemma \ref{la1} gives 
the commutativity and associativity.
\end{proof}

We now fix an even binary code $C\subset \Z_2^r$ and assume that $e,d\in C^{\perp}.$  
\begin{lem}\label{l9.5} Let $x,y\in \Z_2^r.$ Then the restriction of $Y$ to $V(H,d,C+x)\subset V(H,d)$ gives an intertwining operator of type
 $\left(\begin{array}{c}
V(H,d+e, C+x+y)\\
V(H,d, C+x)\ V(H,e,C+y)
\end{array}\right)$
for the  vertex operator algebra $M_C.$
Moreover, the restriction of $Y$ to $V(H_\D,d, C+x)$ gives a linear map
$$Y:  V(H_\D,d, C+x)\to (\Hom(V(H_\D,e,C+y), V(H_\D,d+e, C+x+y))[[z^{1/8}, z^{-1/8}]].$$
\end{lem}
\begin{proof} The result is an immediate consequence of Lemmas \ref{l9.0}, \ref{l9.3} and
 \ref{l9.4}.
\end{proof}

We now assume that $M_C(d,x,\lambda),$ $M_C(e,y,\mu)$ and $M_C(d+e,x+y,\nu)$
 are simple currents with integral conformal weights such that  the space of intertwining operators of type $\left(\begin{array}{c}
 M_C(d+e, x+y, \nu)\\
 M_C(d, x,\lambda)\ M_C(e,y,\mu)
 \end{array}\right)$ is 1-dimensonal.
  \begin{cor}\label{c9.7} 
  	There  is a nonzero intertwining operator of type
 $$\left(\begin{array}{c}
M_C(d+e, x+y, \nu)\\
M_C(d, x,\lambda)\ M_C(e,y,\mu)
\end{array}\right)$$
 such that  for any $u\in M_C(d,x,\lambda)_\D$,  $v\in M_C(e,y,\mu)_\D$ and  $n\in \Z,$
$u_nv\in M_C(d+e,x+y,\nu)_\D.$ 
\end{cor}
\begin{proof}  By  Remark \ref{r0}, we can assume that $W_1=M_C(d,x,\lambda)$ is an irreducible $M_C$-submodule of $V(H,d,C+x),$ $W_2=M_C(e,y,\mu)$ is an irreducible $M_C$-submodule of  $V(H,e,C+y).$ 
Let $Y$ be the intertwining operator given  in Lemma \ref{l9.5}. Then the span of $u_nv$ for $u\in W_1$, $v\in W_2$ and $n\in\Z$ is an irreducible $M_C$-submodule $W$ of
$V(H,d+e,C+x+y).$  Moreover, if $u\in (W_1)_\D, v\in (W_2)_\D$ then $u_nv\in M_C(H_\D, d+e, C+x+y, \nu).$
Again by Remark  \ref{r0}, there exists an irreducible $M_C$-submodule $W_3$ of $M_C(H, d+e, C+x+y, \nu)$
isomorphic to  $M_C(d+e, x+y, \nu)$ such that the projection $P$ of $W$ to $W_3$ is an $M_C$-module isomorphism. Clearly, the composition $Y_1=P\circ Y$ is  a nonzero  intertwining operator 
of type 
 $\left(\begin{array}{c}
M_C(d+e, x+y, \nu)\\
M_C(d, x,\lambda)\ M_C(e,y,\mu)
\end{array}\right).$ 
Using Remark \ref{r0}, we immediately see that if  $u\in M_C(d,x,\lambda)_\D$,   $v\in M_C(e,y,\mu)_\D$ and $n\in \Z$, 
then $u_nv\in M_C(d+e,x+y,\nu)_\D.$ 
\end{proof}

\section{The $\D$-forms of FVOAS}

In this section, we present another main result of this paper. Namely, any framed vertex operator algebra over $\C$ has a $\D$-form, where $\D=\Z[\frac{1}{2}]$.
As a result, we can obtain a framed vertex operator algebra over any algebraically closed field whose characteristic is different from $2, 7$ from any framed vertex operator algebra over $\C.$  

Let $V$ be a framed vertex operator algebra over $\C.$ It follows from \cite{DGH} that 
$$V=\oplus_{d\in D}V^{d}$$
and $V^{0}=M_C$ is a code vertex operator algebra for some binary even codes $C,D\subset \Z_2^r$ with $D\subset C^{\perp}.$ It follows from \cite{DGH}, \cite{LY2}, \cite{DJX} that each $V^d$ is a simple current for the vertex operator algebra $M_C.$ According to \cite{LY2}, we know that for each $d\in D,$ $V^d$ is isomorphic to some $M_C(d,x(d), \lambda(d))$
for some $x(d)=(x(d)_1,...x(d)_r)\in \Z_2^r$  and $\lambda(d).$ For short we will denote $M_C(d,x(d), \lambda(d))$
by $M_C(d,x(d))$ in what follows. We will use ${\cal Y}$ to denote the intertwining operators
defined in Corollary \ref{c9.7}.  

\begin{lem}\label{fform1} If $V$ is a framed vertex operator algebra over $\C$ such that $D(V)$ is isomorphic to $\Z_2,$ then $V$ has a $\D$-form $V_{\D}.$ 
\end{lem}
\begin{proof} Assume $V=M_C\oplus M_C(d,x)$ for some $d\in C^{\perp},$ $x\in\Z_2^r.$ It follows from Corollary \ref{c9.7} that
${\cal Y}(u,z)v\in (M_C)_\D\oplus M_C(d,x)_\D[[z,z^{-1}]]$ for $u,v\in (M_C)_\D\oplus M_C(d,x)_\D.$ On the surface, $(M_C)_\D\oplus M_C(d,x)_\D$ is a perfect candidate for the $\D$-form of $V.$ 
The problem 
is that we do not know the $Y$ which defines the vertex operator algebra structure on $V$ is equal to ${\cal Y}.$ From the definition of ${\cal Y}$ and the skew symmetry we know that $Y(u,z)v={\cal Y}(u,z)v$ for 
$u\in M_C,$ or $u\in M_C(d,x)$ and $v\in M_C.$ Now we assume that $u,v\in M_C(d,x).$ Since $M_C(d,x)$ is a simple current, there exists a nonzero complex number $a$ such that $Y(u,z)v=a{\cal Y}(u,z)v$ for any
$u,v.$ We now set $V_\D=(M_C)_\D\oplus a^{-1/2}M_C(d,x)_\D.$ It is clear that $V_\D$ is a free module over $\D.$ To prove that
$V_\D$ is a $\D$-form, it is good enough to show that $Y(u,z)v\in V_\D[[z,z^{-1}]]$ for $u,v\in V_\D.$ From the discussion before, if $u\in (M_C)_\D,$ or $u\in M_C(d,x)_\D$ and $v\in (M_C)_\D$ this is clear.  Now we let $u=a^{-1/2}u^1, v=a^{-1/2}v^1\in a^{-1/2}M_C(d,x)_\D.$
Then $Y(u,z)v={\cal Y}(u^1,z)v^1\in V_\D.$ 
\end{proof} 

We certainly believe that the constant $a$ in the proof of Lemma \ref{form1} is $1.$ But we cannot prove it in this paper. 

We now can deal with the general case. 

\begin{thm} Any framed vertex operator algebra $V$ over $\C$ has a $\D$-form $V_\D.$
Moreover, 
for any algebraically closed field $\F$ whose characteristic is different from $2,7,$ $V_\F=\F\otimes_\D V_\D$ is a framed vertex operator algebra over $\F.$
\end{thm}
\begin{proof}  Let $D=D(V).$ Note that  for any $d,e\in D,$ $u\in M_C(d, x(d)),$ $v\in M_C(e,x(e))$ and $n\in \Z,$ 
$$Y(u,z)v\in M_C(d+e,x(d+e))[[z,z^{-1}]].$$ 
By Lemmas \ref{l9.0} and \ref{l9.4}
$${\cal Y}(u,z)v\in V(H,d+e, x(d)+x(e)+C)[[z,z^{-1}]]$$ 
Since both  $M_C(d,x(d))$ and $M_C(e,x(e))$ are simple currents, this implies that $M_C(d+e,x(d+e))$ is an irreducible
$M_C$-submodule of $V(H,d+e, x(d)+x(e)+C)$ and $x(d)+x(e)+C=x(d+e)+C.$ So 
$V(H,d+e, x(d)+x(e)+C)=V(H,d+e,x(d+e)+C).$ 

Now let $\{d^1,\dots,d^k\}$ be  a least generating set of $D. $ Then 
$$V=\bigoplus_{1\leq i_1<\cdots< i_s\leq k}V^{\sum_{p=1}^sd^{i_p}}=\bigoplus_{ 1\leq i_1<\cdots< i_s\leq k}M_C(\sum_{p=1}^sd^{i_p},x(\sum_{p=1}^sd^{i_p})).$$
Using the commutativity and associativity of vertex operators
(see Propositions 4.5.7 and 4.5.8 of \cite{LL}) we see that 
 $M_C(\sum_{p=1}^sd^{i_p},x(\sum_{p=1}^sd^{i_p}))$ is spanned by
$u^{i_1}_{n_1}\cdots u^{i_s}_{n_s}\1$ for $u^{i_p}\in M(d^{i_p},x(d^{i_p}))$ and $n_p\in \Z.$ 

According to Lemma \ref{fform1}, for each $d^j,$ there exists a nonzero constant $a_j$ such that 
$(M_C)_\D+a_jM_C(d^j,x(d^j))_\D$ is a $\D$-form of $M_C+M_C(d^j,x(d^j)).$  For short, we set
$V^0_\D=(M_C)_\D$ and $V^{d^j}_\D=a_jM_C(d^j,x(d^j))_\D.$  For any subset $X,Y$ of $V$,  we denote 
the $\D$-span of $x_ny$ for $x\in X$ $y\in Y$ and $n\in \Z$ by $X\cdot Y.$  
For any $d\in D$, there exists $1\leq i_1<\cdots< i_s\leq k$ such that $d=d^{i_1}+\cdots +d^{i_s}.$ 
Set $V^{d}_\D=V^{d^{i_1}}_\D\cdots V^{d^{i_{s-1}}}_\D\cdot V^{d^{i_s}}_\D.$ 
We also set
$$V_\D=\bigoplus_{d\in D}V^d_\D.$$
Clearly, $V=\C\otimes_{\D}V_\D.$ 
We claim $V_\D$ is a $\D$-form of $V.$ 

First we prove that each $V^{d}_\D$ is a free $\D$-module.  We need a general result that  for
$d,e\in D,$ $0\ne a,b\in \C,$
the $\D$-span of $u_nv$  for $u\in a M_c(d,x(d))_\D,$ $v\in bM_C(e,x(e))_\D$ and $n\in \Z$ is contained in 
$cM_C(d+e,x(d+e))_\D$ for some $0\ne c\in \C.$ 
Since $Y(u,z)v=f{\cal Y}(u,z)v$ for some nonzero constant $f$ which depends on $d$ and $e$ only, and
${\cal Y}(u,z)v\in abM_C(d+e,x(d+e))_\D[[z,z^{-1}]]$ we see that 
$u_nv\in fab M_C(d+e,x(d+e))_\D.$ Now let $d=d^{i_1}+\cdots +d^{i_s}.$ 
Then  there exists a nonzero constant $\alpha\in\C$ such that
 $V^{d}_\D$ is a $\D$--submodule of $\alpha M_C(d,x(d))_\D.$ In particular, 
$V^{d}_\D$ is a free $\D$-module.

We finally prove that  $V_\D$ is a vertex operator algebra over $\D.$  Observe that for $u,v,w\in V,$ $m,n\in \Z,$
$$Y(u_mv,z)=\Res\{(z_1-z)^mY(u,z_1)Y(v,z)-(-z+z_1)^mY(v,z)Y(u,z_1)\},$$
$$u_mv_nw=\sum_{i}\alpha_i(u_{s_i}v)_{t_i}w,\ u_mv_nw=\sum_{j}\beta_jv_{p_j}u_{q_j}w$$
for some integers $\alpha_i,\beta_j.$  

Now let  $d=d^{i_1}+\cdots +d^{i_s},$  $e=d^{j_1}+\cdots +d^{j_t}\in D.$ 
Then $V^d_\D=V^{d^{i_1}}_\D\cdots V^{d^{i_{s-1}}}_\D\cdot V^{d^{i_s}}_\D$ and
$V^e_\D=V^{d^{j_1}}_\D\cdots V^{d^{j_{t-1}}}_\D\cdot V^{d^{j_t}}_\D.$ 
 Using the three relations above,   we see that 
$$V^d_\D\cdot V^e_\D\subset V^{d^{i_1}}_\D\cdots  V^{d^{i_s}}_\D\cdot V^{d^{j_1}}_\D\cdots V^{d^{j_{t-1}}}_\D\cdot V^{d^{j_t}}_\D.$$ 
Note that for any $j,$ 
$V^0_\D\cdot V^{d^j}_\D\subset V^{d^j}_\D$ and $V^{d^j}_\D\cdot V^{d^j}_\D\subset V^0_\D$ we immediately see that
$V^d_\D\cdot V^e_\D$ is contained in $V^{d^{l_1}}_\D\cdots V^{d^{l_{q-1}}}_\D\cdot V^{d_{l_q}}_\D$
where $d+e=d^{l_1}+\cdots+ d^{l_q}.$
This shows that  $V_\D$ is a vertex operator algebra over $\D,$ as desired. \end{proof}

\end{document}